\documentclass[11pt]{article}

\usepackage[margin=1.1in]{geometry}
\usepackage{times}
\usepackage[numbers]{natbib}
\usepackage{hyperref}
\usepackage{url}
\usepackage{amsmath, amssymb,amsfonts}
\usepackage{graphicx}
\usepackage{booktabs}

\usepackage{amsmath,amssymb,amsthm,mathtools}
\usepackage{enumitem}
\usepackage{graphicx}
\usepackage{booktabs}
\usepackage{array}
\usepackage{microtype}
\usepackage{xcolor}
\usepackage{url}
\newtheorem{theorem}{Theorem}[section]
\newtheorem{proposition}[theorem]{Proposition}
\newtheorem{lemma}[theorem]{Lemma}
\theoremstyle{definition}
\newtheorem{assumption}[theorem]{Assumption}
\newtheorem{remark}[theorem]{Remark}

\newcommand{\R}{\mathbb{R}}
\newcommand{\E}{\mathbb{E}}
\newcommand{\N}{\mathcal{N}}
\newcommand{\cF}{\mathcal{F}}
\newcommand{\cK}{\mathcal{K}}
\newcommand{\crit}{\operatorname{crit}}
\newcommand{\dist}{\operatorname{dist}}
\newcommand{\Id}{\operatorname{Id}}
\newcommand{\eps}{\epsilon}

\hypersetup{
    colorlinks=true,
    linkcolor=linkcolor,
    citecolor=linkcolor
}

\newcommand{\id}{\operatorname{Id}}

\newcommand{\MMSE}{\mathrm{MMSE}}

\definecolor{linkcolor}{RGB}{82, 82, 192}
\hypersetup{
    colorlinks=true,
    linkcolor=linkcolor,
    citecolor=linkcolor
}

\newcommand{\prox}{\operatorname{Prox}}

\title{Convergent Plug-and-Play Image Restoration with Annealed Noise Levels}

\author{
  Samuel Hurault \\
  LIGM, CNRS, Univ. Gustave Eiffel \\
   \texttt{samuel.hurault@univ-eiffel.fr}
}

\date{}

\begin{document}

\maketitle

\begin{abstract}
Plug-and-Play (PnP) methods solve imaging inverse problems by incorporating deep denoisers into iterative optimization algorithms. Although practical implementations often decrease the denoiser noise level $\sigma$ along iterations, most existing convergence analyses assume a fixed denoiser. In this work, we establish convergence guarantees for a broad family of Plug-and-Play algorithms with annealed noise level, spanning deterministic methods (RED--GD and PnP--PGD) and stochastic methods (SNORE, equivariant RED, and a variant of PnP--Flow). For each method, we identify an explicit, nonconvex objective associated with the terminal denoising level and prove asymptotic stationarity of the iterates with respect to this objective. Our analysis does not prescribe any decay rate for the noise schedule, and our assumptions cover both learned gradient-step denoisers and exact MMSE denoisers. Overall, our theoretical results bridge the gap between existing PnP convergence theory and the decreasing-denoising practices used by state-of-the-art image restoration methods. We empirically demonstrate the benefits of such schedules and illustrate the predicted convergence behavior on several imaging inverse problems, including inpainting, super-resolution, demosaicing and tomography. Code is available at \url{https://github.com/samuro95/annealed-pnp}
\end{abstract}.

\section{Introduction}

Imaging inverse problems aim to recover an unknown image from incomplete or corrupted measurements. These tasks are typically ill-posed, motivating the use of prior information through regularization. A standard variational approach solves
\begin{equation}
\label{eq:global_pb}
 \min_{x\in\R^d} f(x)+\lambda g(x),
\end{equation}
where $f$ measures consistency with the observations, $g$ encodes an image
prior, and $\lambda>0$ balances the two terms. In a Bayesian formulation,
choosing the negative log-likelihood $f = - \log p (y|x)$ and the negative log-prior $g = - \log p$ yields maximum
a posteriori (MAP) estimation. The central challenge is to design a prior that captures the complexity of natural images while retaining an optimization problem solvable with a converging algorithm.
Plug-and-Play (PnP) methods~\citep{venkatakrishnan2013} replace the proximal
operator of the regularizer in such optimization algorithms by a Gaussian
denoiser $D_\sigma$, trained at noise level $\sigma$. Regularization by
Denoising (RED)~\citep{romano2017little} instead uses the residual
$x-D_\sigma(x)$ in place of the regularization gradient. Both PnP and RED algorithms alternate between a data-consistency update and a denoising operation, allowing the same pretrained model to serve as a prior across different inverse problems; see~\citet{hurault2023convergent} for a review. More recently, the use of deep denoisers as learned priors has been further
popularized by diffusion models for sampling instead of optimization.

The parameter $\sigma$ of the plugged denoiser acts as a regularization parameter and changes the fixed point of the plug-and-play algorithm. In practice, it is well established that progressively decreasing the denoiser noise level $\sigma$ over the iterations can substantially improve restoration quality~\citep{zhang2017learning,wei2020tuning,zhang2021dpir}. A large $\sigma$ helps recover coarse structures early in the iterations, while a smaller $\sigma$ allows finer details to emerge later. By contrast, a fixed noise level imposes a trade-off, as illustrated in Figure~\ref{fig:intro}: a small $\sigma$ may fail to recover large structures, whereas a large $\sigma$ may oversmooth fine details. Decreasing $\sigma$ over the iterations combines these two regimes and enables a more accurate multiscale reconstruction. This continuation scheme can also ease optimization: at large $\sigma$, the associated objective is smoother and may be easier to optimize; decreasing $\sigma$ then gradually restores the finer structure of the objective at the terminal noise level~\citep{pesme2025map}. Noise annealing is also central to diffusion sampling models~\citep{song2021score, dhariwal2021diffusion}, whose denoisers are trained over a wide range of noise levels. Recent PnP methods leverage these large denoisers with decreasing noise schedules to achieve state-of-the-art restoration performance~\citep{mardani2024variational,martin2025pnp,park2026sgpnp}.

\begin{figure}[t]
\centering
\includegraphics[width=\linewidth]{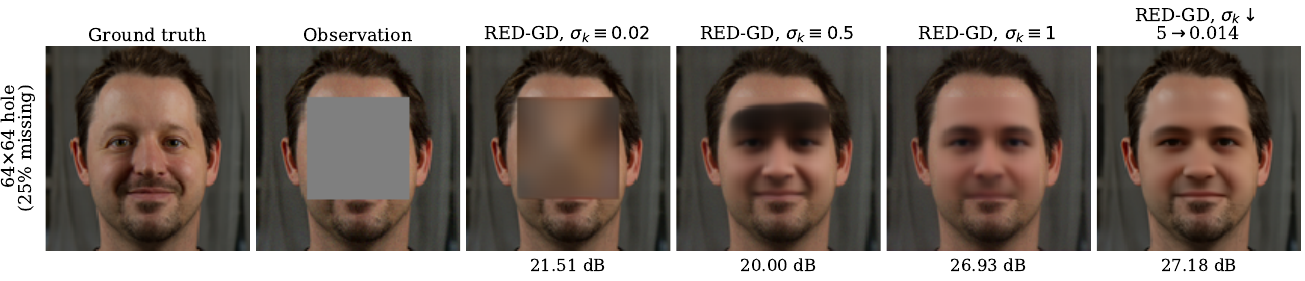}
\caption{Inpainting of a $64\times64$ hole in a $128\times128$ face image using RED gradient descent with the same denoiser (GS--DiffUNet; see Section~\ref{sec:experiments}). With a constant noise level, a small $\sigma$ leaves the initial flat patch, while larger values fill the hole at the cost of missing fine details. Decreasing geometrically $\sigma_k$ along iterates, from a large initial noise level to a small terminal one, produces a sharp face.
}
\label{fig:intro}
\end{figure}

Despite its practical success, most existing convergence analyses still assume a fixed denoiser throughout the iterations, even when decreasing noise schedules are used in experiments. These analyses typically provide fixed-point guarantees under suitable Lipschitz assumptions on the plugged denoiser~\citep{chan2017pnp,ryu2019,sun2019online,sun2021incremental,kamilov2023pnp} or optimization guarantees for structured denoisers~\citep{xu2020mmse,hurault2022gradient,hurault2022proximal,hurault2023relaxed}. Only a few existing convergence analyses handle varying the denoising level. Among them, analyses of annealed PnP--ADMM~\citep{chan2017pnp,hou2022truncated,shrestha2026taming} provide fixed-point guarantees, but under restrictive assumptions on the plugged denoiser and without identifying an underlying variational objective. In a variational setting, \citet{pesme2025map} anneal MMSE denoisers in an inner loop to recover MAP denoising steps, but their strategy does not address single-loop PnP algorithms. Finally, \citet{martin2025pnp} establish convergence of PnP--Flow, a stochastic variant of PnP--PGD, under summable stepsizes, without characterizing the limit.

In this work, we analyze the convergence of a broad class of deterministic and stochastic plug-and-play schemes in which the denoiser noise level monotonically decreases. We prove, for all algorithms, asymptotic stationarity for an explicit variational objective. The deterministic methods are RED gradient descent (RED--GD) and PnP proximal gradient descent (PnP--PGD). The stochastic ones are SNORE~\citep{renaud2024snore}, Equivariant RED~\citep{terris2024, renaud2025equivariant}, and a decoupled variant of SNOPnP~\citep{renaud2026equivariant} and PnP--Flow~\citep{martin2025pnp}.

To retain a variational interpretation, we consider denoisers that are exact gradient or proximal maps. This allows us to interpret PnP and RED methods as real optimization algorithms for minimizing explicit objectives of the form \eqref{eq:global_pb}.
The minimum mean square error (MMSE)
Gaussian denoiser provides a canonical example: by Tweedie's formula~\citep{efron2011tweedie} it is conservative and it is also a proximal operator~\citep{gribonval2011should}.
Although deep denoisers are trained to approximate the MMSE denoiser, they are not conservative in practice
\citep{reehorst2018clarifications}. Gradient-step denoisers re-impose the conservative property by parametrizing the denoiser as
$D_\sigma=\Id-\nabla g_\sigma$, where $g_\sigma$ is a learned scalar
potential~\citep{hurault2022gradient}. Under an additional contractivity condition on $\id - D_\sigma$, the denoiser also admits a proximal representation~\citep{hurault2022proximal}.

Building on this variational structure, we establish convergence guarantees for all annealed algorithms when the noise level $\sigma_k$ is nonincreasing and converges to a floor $\eps>0$. Under algorithm-specific stepsize assumptions, the iterates become asymptotically stationary for an explicit objective associated with the terminal noise level $\epsilon$. 
Our stationarity guarantees require neither convexity of the data term or denoiser potential nor a prescribed annealing rate: any monotone schedule $\sigma_k\downarrow\eps>0$ suffices. The convergence assumptions are verified for both exact MMSE denoisers and state-of-the-art learned conservative denoisers. Experiments on inpainting, demosaicing, super-resolution, and tomography, using denoisers finetuned as gradient-step or proximal denoisers, support the predicted convergence behavior and show consistent gains from annealing over constant-noise baselines.

\section{Various PnP and RED algorithms}
\label{sec:algos}

In this section, we review the different algorithms that will be theoretically analyzed in section~\ref{sec:convergence}. There are two families of Plug-and-Play methods, which we call RED~\citep{romano2017little} and PnP~\citep{venkatakrishnan2013}, that are both built by incorporating a Gaussian denoiser $D_\sigma$ into iterative optimization algorithms.  For each family, we study both deterministic and stochastic variants, when the noise level $\sigma$ of the denoiser decreases along the iterations.

\paragraph{RED algorithms~\citep{romano2017little}.} They are built by replacing the gradient of the prior term $\nabla g(x)$ with a denoising residual $x-D_\sigma(x)$ in gradient-based optimization schemes for solving~\eqref{eq:global_pb}. Applying this substitution to gradient descent, with iteration-dependent noise levels $\sigma_k$ yields:
        \begin{equation}
         \text{(RED--GD)}\qquad
         x_{k+1}=x_k-\gamma\Bigl[\nabla f(x_k)
           +\lambda\bigl(x_k-D_{\sigma_k}(x_k)\bigr)\Bigr].
         \label{eq:redgd}
        \end{equation}
        \citet{renaud2024snore} proposed SNORE, a stochastic version of this algorithm that adds Gaussian noise at each iteration, so that the denoiser is applied to images carrying the level of noise it was trained for. The authors study its convergence for constant $\sigma$ only. With a varying $\sigma_k$, the update is:
        \begin{equation}
     \text{(SNORE)}\qquad
     x_{k+1}=x_k-\gamma_k\Bigl[\nabla f(x_k)
       +\lambda\bigl(x_k-D_{\sigma_k}(x_k+\sigma_kZ_{k+1})\bigr)\Bigr] \qquad Z_{k+1} \overset{\mathrm{i.i.d.}}{\sim} \mathcal{N}(0, \id)
     \label{eq:snore}
    \end{equation}
These Gaussian perturbations of the iterates have then been extended to other random transformations, such as rotations, reflections, and permutations, in the Equivariant RED (ERED) algorithm~\cite{terris2024,renaud2025equivariant}. This aims to make reconstructions robust to transformations under which the image prior is expected to be invariant. 

We analyze in Section~\ref{sec:convergence} the convergence of these RED algorithms, under the assumption that the denoiser $D_\sigma$ is conservative i.e. its residual is the gradient of a smooth potential $g_\sigma : \R^d \to \R$:
\begin{equation}
 D_\sigma(z)=z-\nabla g_\sigma(z).
 \label{eq:gs-denoiser}
\end{equation}
The MMSE Gaussian denoiser $D^{\MMSE}_\sigma(z)=\E[X | X+\sigma Z=z]$, which is the optimal denoiser when training a neural network by minimizing the $L^2$ loss, satisfies this property with $g_\sigma = - \sigma^2 \log p_\sigma$  (Tweedie's formula~\citep{efron2011tweedie}). Another example is the Gradient-Step (GS) denoiser \citep{hurault2022gradient} which is trained explicitly in this form, with $g_\sigma$ parametrized by a neural network. In the following, we focus on these two classes of denoisers.
    
    \paragraph{PnP algorithms~\citep{venkatakrishnan2013}.} They are built by replacing the proximal operator of the regularization term $\prox_{ \tau g}(x)$ by the denoiser $D_\sigma$ in proximal optimization schemes for solving~\eqref{eq:global_pb}. For instance, starting from Proximal Gradient Descent, we get the PnP--PGD algorithm:
    \begin{equation}
     \text{(PnP--PGD)}\qquad
     x_{k+1}=D_{\sigma_k}\Big(x_k - \gamma \nabla f(x_k)\Big).
     \label{eq:pnppgd}
    \end{equation}
    Similarly to SNORE for GD, a stochastic version of this PGD algorithm has been proposed independently by \citet{renaud2026equivariant} (SNOPnP) and by \citet{martin2025pnp} (PnP--Flow, see Appendix~\ref{app:flow-change-variables}).
    
For the convergence of these PnP algorithms, we will suppose that the plugged denoiser $D_\sigma$ is a proximal map i.e. there is a potential $\phi_\sigma : \R^d \to \R \cup \{ + \infty \}$ such that:
\begin{equation}
 D_\sigma(z)= \prox_{\phi_\sigma}(z)  \in \mathop{\arg\min}_{x\in\R^d}
 \left\{\frac12\|x-z\|^2+\phi_\sigma(x)\right\}.
 \label{eq:prox-denoiser}
\end{equation}
The exact MMSE Gaussian denoiser is proximal for the following nonconvex potential (see Lemma~\ref{lem:mmse-conjugate-potential} for the proof) \citep{gribonval2011should}:
\begin{equation}
\label{eq:phi_sigma}
 \phi_\sigma(x)=\left(\frac12\|\cdot\|^2-g_\sigma\right)^*(x)-\frac12\|x\|^2.
\end{equation}
This is also the case of any conservative denoiser~\eqref{eq:gs-denoiser} with contractive $\nabla g_\sigma$~\citep{gribonval2020characterization}. In order to make a learned gradient-step denoiser~\eqref{eq:gs-denoiser} proximal, \citet{hurault2022proximal} proposed a regularization of the training loss to force the contractivity of $\nabla g_\sigma$.

Before turning to their convergence analysis, we clarify that each algorithm includes as parameters: the denoiser noise level $\sigma_k$, which decreases along the iterations, the stepsize $\gamma_k>0$, which may vary along the iterations, and a fixed regularization weight $\lambda$ that balances the regularization term against the data-fidelity term $f$. 

\section{Convergence analysis with annealed noise levels $\sigma_k$} 
\label{sec:convergence}

In this section, we establish convergence guarantees for the algorithms introduced above, when the denoising level $\sigma_k$ decreases along iterates.  Because regularity of the denoiser can deteriorate as
$\sigma\to0$, noise annealing does not decrease below a small positive floor $\epsilon$.  Throughout the analysis, we write
$I_\epsilon=[\epsilon,\sigma_0]$. 
For all methods, monotone convergence of $\sigma_k$ suffices for stationarity, we thus keep the following assumption.
\begin{assumption}[Noise-level schedule]
\label{ass:noise-schedule}
The denoiser levels satisfy $\sigma_{k+1}\leq\sigma_k$ and
$\sigma_k\to\epsilon>0$.
\end{assumption}
We assume that RED algorithms use the conservative denoiser~\eqref{eq:gs-denoiser}
and PnP algorithms its proximal version~\eqref{eq:prox-denoiser}.
Each iterate is then associated with an explicit moving objective consisting of the data-fidelity term $f$ and a
regularization term that depends on the current denoising level.
No convexity assumption is imposed on these objectives since the regularizers induced by gradient-step and MMSE denoisers are typically nonconvex.%
We take the following general assumption on the data-fidelity term.
\begin{assumption}[Data-fidelity term]
\label{ass:prox-lipschitz-x}
The function $f\in\mathcal C^1(\R^d)$ is bounded from below, subanalytic and has
globally $L_f$-Lipschitz gradient.
\end{assumption}
This assumption is verified for a large class of inverse problems, in particular when the observation noise is Gaussian $y \sim \mathcal{N}(Ax, \nu^2 \id) $, for which the data-fidelity term  is the $L^2$ norm $f = \frac{1}{2}\| Ax - y \|^2$. Subanalyticity is a mild sufficient condition for verifying the Kurdyka--\L{}ojasiewicz condition~\citep{bolte2007lojasiewicz}, which is a standard assumption for nonconvex convergence.

In the following subsections, we prove convergence of the different algorithms introduced in Section~\ref{sec:algos}.
For all algorithms, we prove (almost sure) asymptotic stationarity of the iterates relative to an explicit nonconvex objective. Moreover, for the deterministic algorithms, we also prove point convergence of the iterates when $|\sigma_k-\epsilon|=O(k^{-\alpha})$ for some $\alpha > 1$. Table~\ref{tab:conditions} summarizes the convergence guarantees established for the algorithms considered in this work.

\begin{table}[h]
\vspace{-0.5cm}
\centering
\caption{Convergence guarantees for the algorithms introduced in Section~\ref{sec:algos} with $\sigma_k \downarrow \epsilon$.}
\label{tab:theorem-summary}
\footnotesize
\setlength{\tabcolsep}{1.pt}
\renewcommand{\arraystretch}{1.05}
\begin{tabular}{@{}
  >{\raggedright\arraybackslash}m{.20\linewidth}
  >{\raggedright\arraybackslash}m{.22\linewidth}
  >{\raggedright\arraybackslash}m{.32\linewidth}
  >{\raggedright\arraybackslash}m{.22\linewidth}@{}}
\toprule
\textbf{Method} & \textbf{Objective} &
\textbf{Stepsize conditions} & \textbf{Guarantee} \\
\midrule

RED--GD (Section~\ref{sec:gd-pnp})
& $f+\lambda g_\eps$
& $\gamma(L_f+\lambda L_g)<2$
& Stationarity and point convergence \\

SNORE \newline  (Section~\ref{sec:snore})
& $f+\lambda \E_Z[g_\eps(\cdot +\eps Z)]$
& $\gamma_k(L_f+\lambda L_g)<2$;\newline
  $\sum_k\gamma_k=\infty$, $\sum_k\gamma_k^2<\infty$
& Stationarity a.s. \\[8pt]

ERED  \newline (Appendix~\ref{app:ered})
& $f+\lambda \E_Q[g_\eps(Q \cdot)]$
& $\gamma_k(L_f+\lambda L_g)<2$;\newline
  $\sum_k\gamma_k=\infty$, $\sum_k\gamma_k^2<\infty$
& Stationarity a.s. \\

\addlinespace[2pt]

PnP--PGD (Section~\ref{sec:prox-pgd})
& $f+\gamma^{-1}\phi_\eps$
& $\gamma L_f<1$
& Stationarity and point convergence \\

SNOPnP / PnP--Flow 
(Appendix~\ref{app:flow-change-variables})
& $f+\gamma^{-1}\phi_\eps$
& $\gamma L_f<1$;\newline
  $\sum_k\tau_k^2<\infty$
& Stationarity a.s. \\

\bottomrule
\end{tabular}
\label{tab:conditions}
\end{table}

\subsection{Annealed RED--GD}
\label{sec:gd-pnp}

We first study the convergence of the annealed RED--GD algorithm~\eqref{eq:redgd}. We consider a conservative denoiser $D_\sigma = \id - \nabla g_\sigma$. Each update is then a gradient-descent step for the moving objective $F_{\sigma_k}=f+\lambda g_{\sigma_k}$.  

We require the potential $g_\sigma$ to be, uniformly in $\sigma$, coercive, with Lipschitz gradient, and with controlled growth of its derivative w.r.t $\sigma$. More precisely:
\begin{assumption}[Gradient-Step Denoiser]
\label{ass:denoiser_gs}
The denoiser writes $D_\sigma = \id - \nabla g_\sigma$ where $(x,\sigma)\mapsto g_\sigma(x)$ is jointly $\mathcal C^2$ and subanalytic on $\R^d\times(0,\infty)$. We also assume that $\nabla g_\sigma$ is $L_g$-Lipschitz for all $\sigma\in I_\eps$ and that there are positive constants
$c_0, C_0, M_g$ such that, for all
$x\in\R^d, \sigma\in I_\eps$ 
\begin{equation}
 \begin{gathered}
 g_\sigma(x)\geq c_0\|x\|^2-C_0,\quad
 \text{and} \quad 
 |\partial_\sigma g_\sigma(x)|
 \leq M_g(1+\|x\|^2).
 \end{gathered}
 \label{eq:boundedness-global-bounds}
\end{equation}
\end{assumption}
Appendix~\ref{app:denoiser_gs}  verifies these conditions for the exact MMSE denoiser under a compactly supported prior. This assumption is natural for digital images, whose clean pixel values lie in a bounded range, typically $[0,1]$. For learned gradient-step denoisers, we show that the conditions hold up to a mild regularization that leaves the denoiser unchanged over the relevant pixel range. 

Then, under standard stepsize conditions, we show that monotone convergence of $\sigma_k$ to $\epsilon$ suffices for asymptotic stationarity of the iterates with respect to the terminal objective $F_\epsilon$. For summable polynomial decay, we also prove convergence of the sequence to a critical point of this objective. 

\begin{theorem}[Convergence of RED--GD~\eqref{eq:redgd}, proof in Appendix~\ref{app:proof-gd-pnp}]
\label{thm:gd-pnp}
Suppose Assumptions~\ref{ass:noise-schedule},~\ref{ass:prox-lipschitz-x},~\ref{ass:denoiser_gs}
hold and constant stepsize $\gamma<\frac{2}{L_f+\lambda L_g}$. Then
$\|\nabla F_\epsilon(x_k)\|\to0$,
$\|x_{k+1}-x_k\|\to0$, and every cluster point is critical for
$F_\epsilon$.  In particular, $\dist(x_k,\crit F_\eps)\to0$. Moreover,
\begin{equation}
 \min_{0\leq k\leq K}\|\nabla F_{\sigma_k}(x_k)\|^2
 =\mathcal O\!\left(\frac{1}{K}\right).
 \label{eq:gd-pnp-limits}
\end{equation}
If, in addition, $|\sigma_k-\epsilon|=\mathcal O(k^{-\alpha})$
for some $\alpha>1$,
then the iterates have finite length and converge to a critical point
of $F_\epsilon$.
\end{theorem}

Our proof extends to annealed denoisers the fixed-noise convergence analysis of
\citet{hurault2022gradient}. 
Under the stepsize condition, Assumption~\ref{ass:denoiser_gs} yields bounded iterates. On this bounded set, changing the noise level adds at most $|\sigma_{k+1}-\sigma_k|$ to the descent inequality. This error is summable because $\sigma_k\downarrow\eps$, giving asymptotic stationarity for $F_\eps$ without a rate condition. The inexact Kurdyka--\L{}ojasiewicz theorem of \citet[Theorem~10]{ochs2019inexact} gives full-sequence convergence.

\begin{remark}[Rate for the terminal objective]
The rate~\eqref{eq:gd-pnp-limits} concerns the gradient of the moving
objective $F_{\sigma_k}$. The proof gives the same rate for the terminal
gradient if $\sum_k |\sigma_k-\epsilon|^2<\infty$.
\end{remark}

\subsection{Annealed SNORE}
\label{sec:snore}

We now prove convergence of SNORE~\eqref{eq:snore}, the stochastic variant of RED--GD. %
As detailed in~\cite{renaud2024snore}, the algorithm with fixed $\sigma$ targets minimization of $F_{\sigma} = f + \lambda \tilde g_\sigma$ with regularizer
\begin{equation}
    \tilde g_\sigma(x)=\E_Z[g_\sigma(x+\sigma Z)],
    \qquad Z\sim\mathcal N(0,I_d).
    \label{eq:snore-averaged-potential}
\end{equation}
Using decreasing noise levels $\sigma_k\downarrow\eps$, we prove stationarity with respect to the limiting objective $F_\eps=f+\lambda\tilde g_\eps$. 
We obtain the same asymptotic stationarity conclusion as for
RED--GD in Theorem~\ref{thm:gd-pnp}, now almost surely. The main difference is that, as typical for stochastic optimization algorithms, convergence requires square summable vanishing stepsizes.

\begin{theorem}[Convergence of SNORE~\eqref{eq:snore}. Proof in Appendix~\ref{app:proof-sgpnp}]
\label{thm:snore}

Suppose Assumptions~\ref{ass:noise-schedule},~\ref{ass:prox-lipschitz-x},~\ref{ass:denoiser_gs} hold and that the stepsizes verify
\begin{equation}
 0<\gamma_k\leq\bar\gamma,\qquad (L_f + \lambda L_{g})\bar\gamma<2,\qquad
 \sum_k\gamma_k=\infty,\qquad
 \sum_k\gamma_k^2 <\infty.
 \label{eq:sg-descent-conditions}
\end{equation}
Then, almost surely,  
$\|\nabla F_\eps(x_k)\|\to0$, $\|x_{k+1}-x_k\|\to0$, and every cluster
point is critical for $F_\eps$; in particular,
$\dist(x_k,\crit F_\eps)\to0$, almost surely.  Moreover
\begin{equation}
 \min_{0\leq k\leq K}
 \E\|\nabla F_{\sigma_k}(x_k)\|^2
 =\mathcal O\!\left((\textstyle\sum_{j=0}^{K}\gamma_j)^{-1}\right).
 \label{eq:sg-current-rate}
\end{equation}
\end{theorem}

The proof follows the RED--GD strategy in conditional expectation.
The Gaussian noise adds a variance term of order $\gamma_k^2$, which is handled by the summability
assumption. The Robbins--Siegmund theorem
\citep{robbins1971convergence} 
gives summability of  $\gamma_k\|\nabla F_{\sigma_k}(x_k)\|^2$. We then control the remaining stochastic fluctuations to obtain almost-sure stationarity for $F_\eps$.

\begin{remark}[Normalization by $\sigma^2$]
Following the original SNORE version~\citep{renaud2024snore}, for the exact MMSE denoiser, due to the form of Tweedie's formula, one may divide the denoiser residual in \eqref{eq:gs-denoiser} by $\sigma_k^2$. Theorem~\ref{thm:snore} then applies with
$F_\sigma = f+ \frac{\lambda}{\sigma^2 }\tilde g_\sigma$. Since $\sigma_k \to \epsilon>0$, the only difference is that the stepsize condition becomes $\left(L_f+\lambda L_{g}/\epsilon^2\right) \bar\gamma < 2 $.
\end{remark}
In Appendix~\ref{app:ered}, we also prove the same stationarity result for the stochastic ERED algorithm~\eqref{eq:ered}.

\subsection{Annealed PnP--PGD}
\label{sec:prox-pgd}

We now study the convergence of PnP algorithms, in which the proximal operator of the regularization term is replaced by a denoiser. Our analysis relies on denoisers that can be written as proximal mappings $D_\sigma=\prox_{\phi_\sigma}$ for a potential $\phi_\sigma$. As explained in Section~\ref{sec:algos}, this is verified for both the MMSE denoiser and a learned gradient-step denoiser $D_\sigma=\id-\nabla g_\sigma$ when $\nabla g_\sigma$ is contractive.

\begin{assumption}[Proximal denoiser]
\label{ass:regular-proximal-denoiser}
Let $D_\sigma = \id - \nabla g_\sigma$. Assume either that there exists $\rho<1$ such that, for every
$\sigma\in I_\epsilon$, $\nabla g_\sigma$ is $\rho$-Lipschitz,
or that $D_\sigma$ is the exact MMSE denoiser associated with a
compactly supported prior whose support is not contained in an affine
hyperplane.
\end{assumption}

As detailed in Appendix~\ref{app:regular-proximal-denoiser}, in the MMSE case, compact support ensures that $\operatorname{Im}(D_\sigma) \subseteq \operatorname{conv}(\operatorname{supp}P_X)$
uniformly in $\sigma$. This ensures that every PnP iterate remains bounded. %

With such proximal denoiser, annealed PnP--PGD~\eqref{eq:pnppgd} is proximal gradient descent for the nonconvex moving objective $F_{\sigma_k} :=  f+ \gamma^{-1}\phi_{\sigma_k}$.
The stepsize appears in the objective because the proximal denoiser applies $\operatorname{prox}_{\phi_{\sigma_k}}$ instead of
$\operatorname{prox}_{\gamma \phi_{\sigma_k}}$ as for standard PGD.  The stepsize thus changes the effective regularization weight. %
Similar to RED--GD, we prove convergence of the algorithm and asymptotic stationarity with respect to the terminal objective $F_{\eps} := f+ \gamma^{-1}\phi_{\eps}$.

\begin{theorem}[Convergence of PnP--PGD~\eqref{eq:pnppgd}, proof in Appendix~\ref{app:proof-prox-pgd}]
\label{thm:prox-pgd}
Suppose Assumptions~\ref{ass:noise-schedule},~\ref{ass:prox-lipschitz-x},~\ref{ass:denoiser_gs},~\ref{ass:regular-proximal-denoiser}
hold and constant stepsize $\gamma < \frac{1}{L_f}$. Then, 
$\|\nabla F_\eps(x_k)\|\to0$,
$\|x_{k+1}-x_k\|\to0$, and every cluster point is critical for
$F_\eps$.  Moreover,
\begin{equation}
\min_{1\leq k\leq K}\|\nabla F_{\sigma_k}(x_k)\|^2
 =\mathcal O\!\left(\frac{1}{K}\right).
 \label{eq:prox-pgd-gradient-limits}
\end{equation}
If, in addition, $|\sigma_k-\eps|=\mathcal O(k^{-\alpha})$
for some $\alpha>1$,
then the iterates have finite length and converge to a critical point
of $F_\eps$.
\end{theorem}

Note that, because the stepsize enters the objective $F_\sigma = f + \gamma^{-1} \phi_\sigma$, the condition $\gamma L_f<1$ imposes a lower bound on the effective regularization weight \citep{hurault2023relaxed}. This can be restrictive for PnP--PGD with fixed $\sigma$: in our face-hole experiment (Section~\ref{sec:experiments}), $\sigma$ must be large enough to fill the missing region, but this leads to oversmoothing unless the regularization weight is reduced, i.e., unless $\gamma$ is increased. Annealing precisely mitigates this tradeoff: a large $\sigma$ helps fill the hole early, while a smaller terminal $\sigma$ limits overregularization near convergence.

\paragraph{Annealed PnP--Flow / SNOPnP.}

We also study in Appendix~\ref{app:flow-change-variables} the convergence of a stochastic version of PnP--PGD with decreasing noise-levels. The algorithm appeared in both PnP--Flow~\citep{martin2025pnp} and SNOPnP~\citep{renaud2026equivariant}, and we prove that they are equivalent. In these algorithms, the Gaussian perturbation and the denoiser use the same level $\sigma_k$. When $\sigma_k\to\eps>0$, the injected noise retains nonzero variance, so our descent argument cannot control its accumulated effect. We therefore let the noise amplitude $\tau_k$ vanish independently of the denoiser level $\sigma_k$
\begin{equation}
 x_{k+1}=D_{\sigma_k}\Bigl(
   x_k-\gamma\nabla f(x_k)+\tau_kZ_{k+1}\Bigr),
 \qquad Z_{k+1}\sim\N(0,I_d).
 \label{eq:snopnp}
\end{equation}
Theorem~\ref{thm:PnP--Flow} in Appendix~\ref{app:flow-change-variables} then proves, for $\sum_k\tau_k^2<\infty$, almost-sure asymptotic stationarity of these iterates with respect to $F_\epsilon$.

\section{Experiments}
\label{sec:experiments}

In the experiments, we illustrate, on various image inverse problems, the convergence properties established in the previous section. We also confirm that decreasing the denoiser noise along the iterations improves restoration quality over running the same algorithm at constant $\sigma$. All experiments are implemented with the DeepInverse library~\citep{tachella2025deepinverse}.

\paragraph{Inverse problems and data.}
We experiment on both natural images ($68$ images of CBSD68~\citep{martin2001database} downscaled to $128\times128$) and face images  ($100$ images from FFHQ~\citep{karras2019style} at $128\times128$). We consider different linear inverse problems with observations
$y=Ax+w$, where $w\sim\mathcal N(0,\sigma_n^2 I)$ ($\sigma_n = 0.01$ across all experiments) (i) inpainting: a centered $64\times64$ hole on faces and a random mask hiding $50\%$ of the pixels on natural images (ii) demosaicing of a Bayer color filter array (iii) $\times4$ super-resolution with a bicubic anti-aliasing filter. We also provide complementary experiments on sparse-view parallel-beam tomography in Appendix~\ref{app:experiments}. The algorithms are initialized from the observed pixels completed by a normalized convolution for inpainting and demosaicing, and bicubic upsampling for super-resolution.

\paragraph{Algorithms.}
We compare RED--GD~\eqref{eq:redgd}, SNORE~\eqref{eq:snore}, Equivariant RED (ERED)~\eqref{eq:ered}, PnP--PGD~\eqref{eq:pnppgd} and SNOPnP~\eqref{eq:snopnp}. For ERED, the random transformations are drawn from the dihedral group (rotations by multiples of $90^\circ$ and flips) on natural images, and from $\{\id,\text{horizontal flip}\}$ on face images. As an external reference we also report DPIR~\citep{zhang2021dpir}, i.e.\ eight half-quadratic splitting iterations with its default decreasing noise schedule.

\paragraph{Denoisers.}
All experiments are run with gradient-step (GS) denoisers $D_\sigma=\id-\nabla g_\sigma$ with, following~\citet{hurault2022gradient} $g_\sigma(x)=\frac12\|x-N_\sigma(x)\|^2$ where $N_\sigma$ is a UNet neural network. The denoiser $D_\sigma$ is trained using the standard $L^2$ denoising loss.  
For experiments on FFHQ faces, we parameterize $N_\sigma$ using the pretrained DDPM U-Net architecture from~\citet{dhariwal2021diffusion} (referred to as \emph{DiffUNet} in the deepinv library) which we finetune with our GS denoiser loss on $128\times128$ images over noise levels $\sigma\in[0.002,5]$. We call it GS--DiffUNet. 
For natural image experiments, we parameterize $N_\sigma$ using the DRUNet architecture~\citep{zhang2021dpir}, the resulting gradient-step denoiser is called \emph{GS--DRUNet}.  Since the original model of~\citet{hurault2022gradient} was trained only for noise levels $\sigma\leq 0.2$, we finetune the denoiser on the extended range $\sigma\in[0.002,2]$. This wider noise-level range is particularly useful for annealed schemes.  

Moreover, we recorded every input passed to a GS denoiser in the reported experiments and verified that all such inputs remain in a fixed compact set. Consequently, the regularization introduced in Appendix~\ref{app:denoiser_gs} is never activated, and Assumption~\ref{ass:denoiser_gs} is satisfied along the reported trajectories.

For the PnP experiments, in order to obtain proximal denoisers, following~\citet{hurault2022proximal}, we finetune (and relax) these GS denoisers with a soft penalty that encourages $\| \nabla^2 g_\sigma (x)\|_S < 1$. The resulting denoiser is then, at least locally, the proximal map of a potential $\phi_\sigma$. We use Lanczos iterations for estimating the spectral norm during training, which is considerably faster to converge than the power iterations initially used by~\citet{hurault2022proximal}. Unlike conservativity, which holds by construction, the contraction condition sufficient for proximality is promoted by a soft penalty, which is not guaranteed for every input, but is expected to hold near the image manifold. See Appendix~\ref{app:denoiser_training} for details.

\paragraph{Schedules and hyper-parameters.}
In all experiments, the noise level follows the geometric schedule 
\begin{equation}
\label{eq:schedule}
\sigma_k=\eps+(\sigma_0-\eps)r^k \quad \text{with} \quad 0<r<1
\end{equation}
which decreases at every iteration and converges to $\eps$ faster than the $\mathcal O(k^{-\alpha})$, $\alpha>1$, required for full sequence convergence of Theorems~\ref{thm:gd-pnp} and~\ref{thm:prox-pgd}.

We use stepsizes which follow each algorithm's requirements for convergence. For RED--GD, SNORE and ERED, the stepsize condition involves the Lipschitz constant of $\nabla g_\sigma$ which is first estimated empirically on a validation set. The stochastic algorithms SNORE and ERED require vanishing stepsizes ($\sum \gamma_k = \infty$ and $\sum \gamma_k ^2 < \infty$), we keep the step constant during the main noise annealing phase and then let it decay polynomially. We tune these method-specific parameters on validation images. Appendix~\ref{app:experiments_algos} gives the exact schedules, and parameter values. %

\paragraph{Results.}
Table~\ref{tab:main} reports final PSNR on FFHQ, while Table~\ref{tab:natural} in the Appendix reports results on natural images (CBSD68). Both tables compare constant and annealed noise levels versions. 

\begin{table}[h]
\vspace{-0.3cm}
\centering\small
\setlength{\tabcolsep}{4.5pt}
\caption{PSNR (dB) on $100$ FFHQ face images after $500$ iterations with constant or annealed~\eqref{eq:schedule} noise level. RED methods and DPIR use the GS--DiffUNet denoiser; PnP methods use Prox--DiffUNet. Best value per problem in bold. The corresponding CBSD68 results are in Table~\ref{tab:natural} of Appendix~\ref{app:experiments}.}
\label{tab:main}

\begin{tabular}{@{}lrrrrrr@{}}
\toprule
 & \multicolumn{2}{c}{\textbf{Inpainting}} & \multicolumn{2}{c}{\textbf{Demosaicing}} & \multicolumn{2}{c}{\textbf{SR $\times 4$}}\\
\cmidrule(lr){2-3}\cmidrule(lr){4-5}\cmidrule(l){6-7}
\textbf{Method} & \multicolumn{1}{c}{Constant} & \multicolumn{1}{c}{Annealed} & \multicolumn{1}{c}{Constant} & \multicolumn{1}{c}{Annealed} & \multicolumn{1}{c}{Constant} & \multicolumn{1}{c@{}}{Annealed}\\
\midrule
RED--GD & 24.22 & 25.36 & 33.54 & 37.89 & 27.17 & 27.51\\
PnP--PGD & 20.03 & 25.09 & 35.35 & 37.76 & 26.49 & 27.25\\
SNORE & 24.34 & 24.98 & 33.32 & \textbf{38.03} & 27.14 & 27.34\\
SNOPnP & 21.93 & 24.42 & 36.22 & 37.78 & 27.18 & 27.29\\
ERED & 24.08 & \textbf{25.44} & 33.55 & 38.00 & 27.24 & \textbf{27.56}\\
\addlinespace[2pt]
DPIR  & \multicolumn{2}{c}{19.83} & \multicolumn{2}{c}{35.90} & \multicolumn{2}{c}{26.26}\\
\bottomrule
\end{tabular}

\end{table}

Figures~\ref{fig:hole}, \ref{fig:demosaic}, and \ref{fig:sr} show representative reconstructions. Annealing improves PSNR in every case, with the largest gains in inpainting and demosaicing. As Figure~\ref{fig:hole} shows, a low constant $\sigma$ can leave large holes unfilled (PnP--PGD), while a high constant $\sigma$ can produce blurry results (RED--GD). Annealing $\sigma$ fills the holes with more detail. In demosaicing, it removes the color artifacts that remain with a constant $\sigma$. Sparse-view tomography also benefits substantially from noise-level annealing (Table~\ref{tab:natural}). It provides smaller gains for inpainting with randomly missing pixels (Table~\ref{tab:natural}) and $\times 4$ super-resolution (Table~\ref{tab:main}). In these settings, no large missing region needs to be reconstructed, so a small, fixed $\sigma$ can work well throughout the iterations.


Table~\ref{tab:main} also shows that annealing is particularly useful for PnP algorithms (PnP--PGD and SNOPnP), in particular  for large hole inpainting. As explained in Section~\ref{sec:prox-pgd}, for these algorithms, the stepsize $\gamma$ also determines the effective regularization weight. Using a large constant $\sigma$ would require increasing $\gamma$ to avoid over-regularization, but the convergence condition $\gamma L_f<1$ limits this adjustment. RED methods do not face this constraint, as their regularization weight can be tuned independently. Noise annealing alleviates the limitation for PnP--PGD by using a large $\sigma$ early to fill the hole and a small $\sigma$ near convergence, thereby avoiding over-regularization in the final reconstruction.

\begin{figure}[h]
\centering
\includegraphics[width=\linewidth]{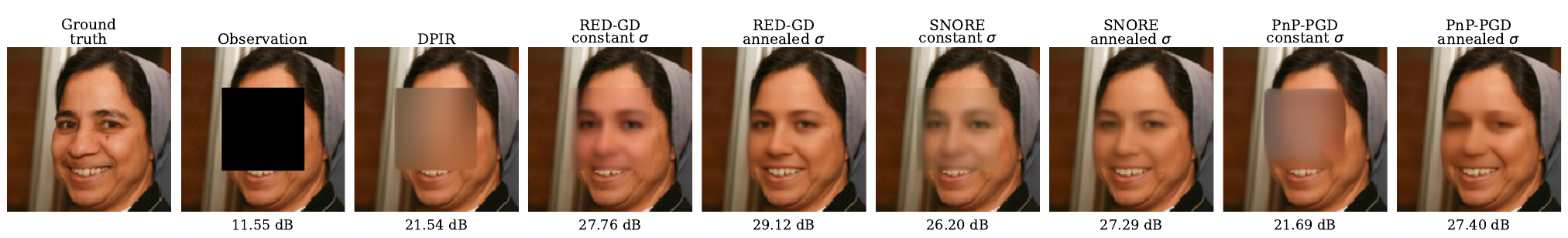}\\[2pt]
\includegraphics[width=\linewidth]{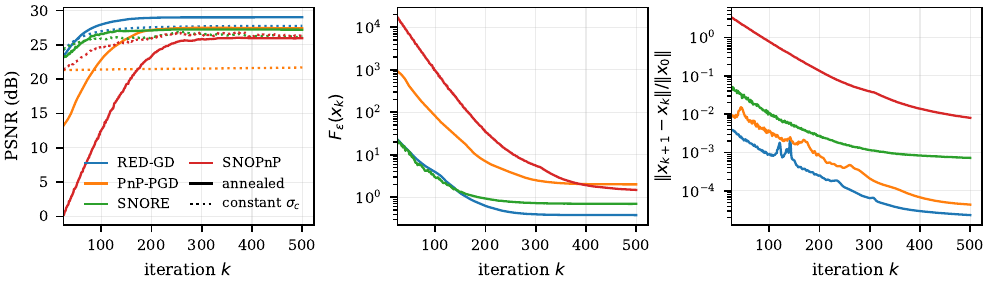}
\caption{Inpainting of a $64\times64$ hole on FFHQ $128\times128$. Top: the constant-$\sigma$ versions (and DPIR) leave a smooth patch, while the annealed versions fill the hole. Bottom: PSNR, target objective $F_\eps(x_k)$ and residual $\|x_{k+1}-x_k\|/\|x_0\|$ of the annealed runs, along the iterations, for the image shown above. Note that every algorithm has its own target objective $F_\eps$, so only the decrease of each curve is meaningful, not the levels.}
\label{fig:hole}
\end{figure}

\begin{figure}[t]
\centering
\includegraphics[width=\linewidth]{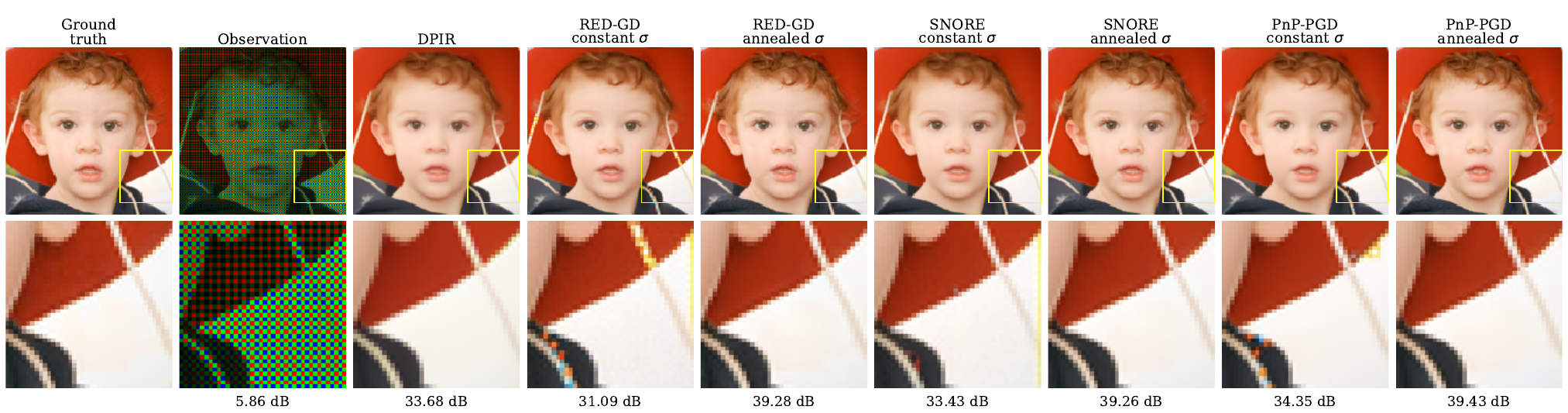}\\[2pt]
\includegraphics[width=\linewidth]{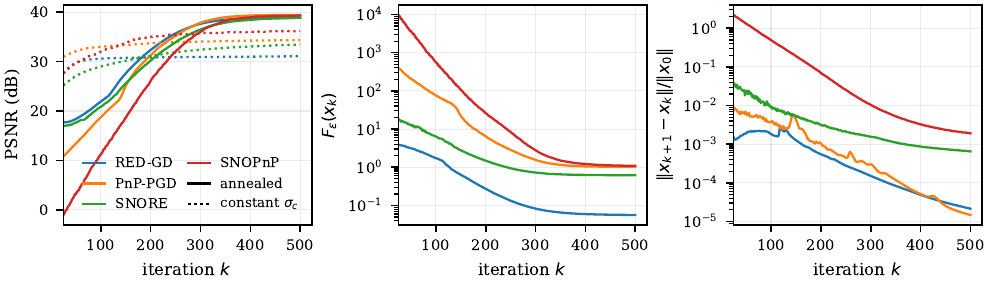}
\caption{Demosaicing of a Bayer color filter array on FFHQ $128\times128$. The constant-$\sigma$ versions (and DPIR) keep the color fringes of the initial interpolation along the edges, the annealed versions remove them. Bottom: PSNR, terminal objective and residual along the iterations.}
\label{fig:demosaic}
\end{figure}

\clearpage
\paragraph{Convergence.}
The bottom rows of Figures~\ref{fig:hole}, \ref{fig:demosaic}, and~\ref{fig:sr} track the PSNR, the terminal objective $F_\eps(x_k)$ associated with each algorithm, and the normalized residual $\|x_{k+1}-x_k\|/\|x_0\|$ for the image shown above. Details on the computation of $F_\eps(x_k)$ are given in Appendix~\ref{app:experiments_algos}. Across the three tasks, as predicted by our theoretical results, the objective values stabilize and the residuals decay toward zero. Note that although annealing can initially yield lower PSNR than a constant-noise run, its PSNR later overtakes the baseline as the noise level decreases.

\begin{figure}[t]
\centering
\includegraphics[width=\linewidth]{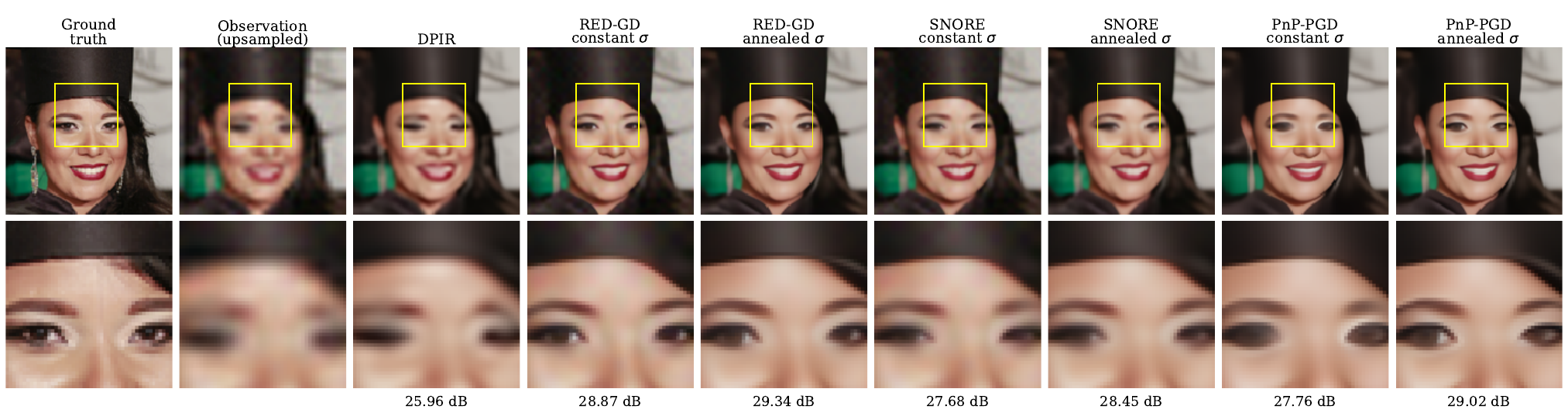}\\[2pt]
\includegraphics[width=\linewidth]{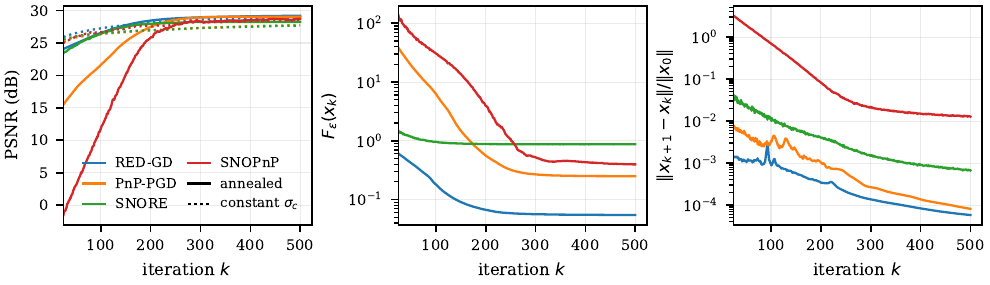}
\caption{$\times4$ super-resolution on FFHQ $128\times128$. PSNR differences are smaller than in the previous experiments, but the annealed methods yield visibly sharper and more detailed reconstructions. Bottom: PSNR, terminal objective and residual along the iterations.}
\label{fig:sr}
\end{figure}
\section{Conclusion}

We have shown theoretically that plug-and-play (RED and PnP) algorithms with decreasing denoising strength converge toward stationary points of an explicit terminal objective. Our analysis does not prescribe an annealing rate or require convexity assumptions. This closes the gap between previous fixed-denoiser PnP convergence theories and the decreasing noise schedules used in modern PnP restoration. We confirm experimentally on various inverse problems that noise-annealed PnP converges and gives better reconstruction than its constant noise level counterpart. 

Our results show that annealing preserves convergence to stationarity, but do not explain why it may yield better solutions: annealed and constant-level algorithms satisfy the same stationarity guarantee, which does not exclude poor critical points. Understanding theoretically how annealing influences stationary-point selection remains an open question. A second limitation is that our analysis requires conservative denoisers and/or contractive residuals. These properties hold for the ideal MMSE denoiser and are enforced or encouraged through regularization for our trained denoisers. However, standard learned denoisers, even if trained to approximate the MMSE denoiser, may satisfy these assumptions only approximately due to finite model capacity and imperfect training. Quantifying how such approximation errors affect the convergence guarantees is an important direction for future work.

\bibliographystyle{plainnat}
\bibliography{references}

\clearpage
\appendix

\section{Verification of assumptions}
\label{app:denoiser-verification}

\subsection{Verification of Assumption~\ref{ass:denoiser_gs}}
\label{app:denoiser_gs}

We check the three global bounds in
Assumption~\ref{ass:denoiser_gs}, as well as joint regularity and
subanalyticity. %

\paragraph{Exact MMSE denoiser with compact prior.}
Assume that \(\|X\|\leq R\) almost surely, and let
\(p_\sigma\) be the density of \(Y=X+\sigma Z\), where
\(Z\sim\mathcal N(0,I_d)\). Set
\(g_\sigma(x)=-\sigma^2\log p_\sigma(x)\).
Gaussian convolution makes \(p_\sigma>0\)
jointly real analytic in \((x,\sigma)\) for \(\sigma>0\).
Thus \(g_\sigma\) is jointly real analytic, hence jointly
\(\mathcal C^\infty\) and subanalytic. Tweedie's formula and
differentiation of the posterior mean give
\[
 \nabla g_\sigma(x)=x-D^{\MMSE}_\sigma(x),\qquad
 \nabla^2g_\sigma(x)
 =I-\sigma^{-2}\operatorname{Cov}(X\mid Y=x).
\]
Since \(\|\operatorname{Cov}(X\mid Y=x)\|_{\rm op}\leq R^2\),
the gradients $\nabla g_\sigma$ are globally \(L_g\)-Lipschitz for all
\(\sigma\in I_\eps\), with
\(L_g=1+R^2/\eps^2\).

To verify coercivity and growth in \(\sigma\), write
\begin{equation}
\label{eq:gsigma}
 g_\sigma(x)=\frac d2\sigma^2\log(2\pi\sigma^2)
              -\sigma^2\log W_\sigma(x).
\end{equation}
with $W_\sigma(x):=\E\exp\!\left(-\frac{\|x-X\|^2}{2\sigma^2}\right)$.
Using
$$
\|x-X\|^2 \ge \tfrac12\|x\|^2-\|X\|^2 \ge\tfrac12\|x\|^2-R^2, 
$$
we have
\[
\log W_\sigma(x)\le-\frac{\|x\|^2}{4\sigma^2}+\frac{R^2}{2\sigma^2}
\]
Thus, for all $\sigma \in I_\epsilon$,
\[
g_\sigma(x)\ge\tfrac14\|x\|^2-\tfrac12R^2 +\tfrac d2\sigma^2\log(2\pi\sigma^2).
\]
Finally, differentiating~\eqref{eq:gsigma} with respect to $\sigma$ gives
\[
 \partial_\sigma g_\sigma(x)
 =d\sigma\bigl(\log(2\pi\sigma^2)+1\bigr)
  -2\sigma\log W_\sigma(x)
  -\frac1\sigma\E[\|x-X\|^2\mid Y=x].
\]
Both \(|\log W_\sigma(x)|\) and the conditional squared distance
are bounded by a constant times \(1+\|x\|^2\), uniformly for
\(\sigma\in I_\eps\). Hence
\(|\partial_\sigma g_\sigma(x)|
 \leq M_g(1+\|x\|^2)\), as required.

\paragraph{Gradient-step denoisers.}

Suppose $g_\sigma:\R^d\to\R$ is parametrized by a
feedforward network with affine layers and $\mathcal C^2$
activations whose first and second derivatives are globally bounded.
Assume these bounds are uniform for $\sigma\in I_\eps$. Then the $\nabla g_\sigma$ are uniformly globally Lipschitz. However, ~\cite{hurault2022gradient} consider learned potential parametrized as
\[
g_\sigma(x)=\frac12\|x-N_\sigma(x)\|^2,
\]
where $(x,\sigma)\mapsto N_\sigma(x)$ is a neural network, with activation chosen such that it is jointly $\mathcal C^2$ in $(x, \sigma)$
and subanalytic. We propose here a mild regularization of this potential such that it is uniformly Lipschitz in $x$ and uniformly coercive. 

Fix a compact box $\mathcal B$ containing the
prescribed image range, typically $\mathcal B = [-1,2]^d$ for image pixels normalized in $[0,1]$.  Choose a $\mathcal C^2$ semialgebraic cutoff
$\chi:\R^d\to[0,1]$ that equals $1$ on a neighbourhood of
$\mathcal B$ and vanishes outside a ball of radius $R$. For any
$\mu>0$, define
\begin{equation}
\widetilde g_\sigma(x)
=\chi(x)^2g_\sigma(x)
 +\frac{\mu}{2}\bigl(1-\chi(x)^2\bigr)\|x\|^2,
\qquad
\widetilde D_\sigma=\Id-\nabla\widetilde g_\sigma .
\label{eq:regularized-gs-tail}
\end{equation}
On $\mathcal B$, both the potential and the denoiser coincide with
their original counterparts. Outside the ball of radius $R$,
$\widetilde g_\sigma(x)=\frac{\mu}{2}\|x\|^2$.

This single modification gives the required global bounds. Since
$g_\sigma\geq0$, for every $x$ and $\sigma\in I_\eps$,
\[
\widetilde g_\sigma(x)
\geq \frac{\mu}{2}\|x\|^2-\frac{\mu R^2}{2}.
\]
Moreover, the $x$-Hessian of $\widetilde g_\sigma$ is uniformly bounded
on the compact support of $\chi$, by joint $\mathcal C^2$ regularity
and compactness of $I_\eps$; outside that support it equals $\mu I$.
Hence $\nabla\widetilde g_\sigma$ is globally Lipschitz, with a
constant uniform in $\sigma\in I_\eps$. Finally,
\[
\partial_\sigma\widetilde g_\sigma(x)
=\chi(x)^2\partial_\sigma g_\sigma(x)
\]
is uniformly bounded on $\R^d\times I_\eps$, and therefore satisfies
the required bound $|\partial_\sigma\widetilde g_\sigma(x)|
\leq M_g(1+\|x\|^2)$. Joint $\mathcal C^2$ regularity and
subanalyticity are preserved. Thus $\widetilde g_\sigma$ verifies
Assumption~\ref{ass:denoiser_gs}. 

This penalty provides a coercive safeguard against unbounded iterates.
It remains inactive whenever the inputs to the denoiser stay inside
$\mathcal B$, in which case the regularized and original iterations
coincide.

\subsection{Regularity properties of the proximal denoisers}
\label{app:regular-proximal-denoiser}

\paragraph{MMSE denoiser.} We first prove in Lemma~\ref{lem:mmse-conjugate-potential} the form \eqref{eq:phi_sigma} of the MMSE proximal potential. Then, in Proposition~\ref{prop:various_mmse_results}, we prove its regularity properties that will be useful in the PnP proofs.

\begin{lemma}[MMSE proximal potential]
\label{lem:mmse-conjugate-potential}
Assume $\E\|X\|^2<\infty$, let $Z\sim\mathcal N(0,I_d)$ be independent
of $X$, and fix $\sigma>0$. Let $p_\sigma$ be the density of
$X+\sigma Z$, define $g_\sigma=-\sigma^2\log p_\sigma$, and write
$D^{\MMSE}_\sigma(z)=\E[X\mid X+\sigma Z=z]$. Then the potential
\begin{equation*}
 \phi_\sigma(x)
 =\left(\frac12\|\cdot\|^2-g_\sigma\right)^*(x)
   -\frac12\|x\|^2
\end{equation*}
is proper and lower semicontinuous, and
$D^{\MMSE}_\sigma(z)=\operatorname{prox}_{\phi_\sigma}(z)$ for every
$z\in\R^d$.
\end{lemma}
\begin{proof}
By \citet[Theorem~1]{gribonval2020characterization}, it suffices for
proximality to show that $D^{\MMSE}_\sigma$ is the gradient of a convex
function. We then identify the potential explicitly.
Expanding the Gaussian density gives
\begin{equation*}
 p_\sigma(z)=(2\pi\sigma^2)^{-d/2}
 e^{-\|z\|^2/(2\sigma^2)}
 \E\exp\!\left(
   \frac{\langle z,X\rangle-\|X\|^2/2}{\sigma^2}\right),
\end{equation*}
and hence
\begin{equation}
 \frac12\|z\|^2-g_\sigma(z)
 =\sigma^2\log\E\exp\!\left(
   \frac{\langle z,X\rangle-\|X\|^2/2}{\sigma^2}\right)
   -\frac d2\sigma^2\log(2\pi\sigma^2).
 \label{eq:mmse-log-partition}
\end{equation}
The right-hand side is a log-partition function of affine functions
of $z$, up to a constant, so it is convex. 
Using Tweedie's formula, we have
$$D^{\MMSE}_\sigma(z) = \nabla\!\left(\frac12\|z\|^2-g_\sigma(z)\right).$$

For the $\phi_\sigma$ defined in the Lemma, we have
\begin{equation*}
 \frac12\|x-z\|^2+\phi_\sigma(x)
 = \left(\frac12\|\cdot\|^2-g_\sigma\right)^*(x)
    -\langle z,x\rangle+\frac12\|z\|^2,
\end{equation*}
By Fenchel-Young inequality, the right-hand side is greater than
\begin{equation*}
\langle z,x\rangle - \frac12\|z\|^2 + g_\sigma(z)  -\langle z,x\rangle+\frac12\|z\|^2 = g_\sigma(z)
 \end{equation*}
with equality if and only if $x = \nabla(\frac12\|\cdot\|^2-g_\sigma)(z) = D^{\MMSE}_\sigma(z)$
\end{proof}

\begin{proposition}[Properties of the MMSE denoiser]
\label{prop:various_mmse_results}
Suppose \(P_X\) has compact support not contained in an affine
hyperplane, and set
\[
 C=\operatorname{conv}(\operatorname{supp}P_X),
 \qquad U=\operatorname{int}C.
\]
Then, for every \(\sigma>0\), \(D^{\MMSE}_\sigma\) is a smooth
diffeomorphism from \(\R^d\) onto 
\(U\), and \(\operatorname{int}(\operatorname{dom}\phi_\sigma)=U\).
Moreover, the maps \((z,\sigma)\mapsto D^{\MMSE}_\sigma(z)\) and
\((x,\sigma)\mapsto\phi_\sigma(x)\) are jointly smooth on
\(\R^d\times(0,\infty)\) and \(U\times(0,\infty)\), respectively,
\(\phi_\sigma\) is real analytic on \(U\), and
\[
 \nabla\phi_\sigma(x)=(D^{\MMSE}_\sigma)^{-1}(x)-x,
 \qquad x\in U.
\]
\end{proposition}
\begin{proof}
Fix \(\sigma>0\), and let
\[
 \nu_\sigma(du)=e^{-\|u\|^2/(2\sigma^2)}P_X(du),
 \qquad
 A_\sigma(\theta)
 =\log\int e^{\langle\theta,u\rangle}\nu_\sigma(du).
\]
The posterior at \(Y=z\) is its exponential tilt with parameter
\(z/\sigma^2\); thus
\[
 D^{\MMSE}_\sigma(z)=\nabla A_\sigma(z/\sigma^2),
 \qquad
 J_zD^{\MMSE}_\sigma(z)
 =\sigma^{-2}\operatorname{Cov}(X\mid Y=z)\succ0.
\]
Each posterior has the same full-dimensional support as \(P_X\);
hence the covariance is positive definite and the mean map is
injective and locally invertible. Its values lie in \(U\): a
posterior mean on the boundary of \(C\) would force the posterior
to be supported on a proper supporting face.

Conversely, if \(x\in U\), finitely many support points have a
convex hull containing a ball around \(x\). Small neighborhoods of
these points have positive \(\nu_\sigma\)-mass and show that, for
some \(c>0\) and \(C_x<\infty\),
\[
 A_\sigma(\theta)-\langle\theta,x\rangle
 \geq c\|\theta\|-C_x.
\]
This strictly convex function attains its minimum, where
\(\nabla A_\sigma(\theta)=x\). Hence
\(\operatorname{Im}(D^{\MMSE}_\sigma)=U\), independently of
\(\sigma\).

Let \(h_\sigma(z)=\frac12\|z\|^2-g_\sigma(z)\).
For \(x\in U\), let \(z=(D^{\MMSE}_\sigma)^{-1}(x)\).
Since \(\nabla h_\sigma(z)=x\), Fenchel equality gives
\begin{equation}
 \phi_\sigma(x)
 =\langle z,x\rangle-\frac12\|z\|^2
   +g_\sigma(z)-\frac12\|x\|^2.
 \label{eq:phi_x_z}
\end{equation}
Thus \(U\subset\operatorname{int}(\operatorname{dom}\phi_\sigma)\).
Compact support also gives
\(h_\sigma(z)\leq\sup_{u\in C}\langle z,u\rangle+C_\sigma\).
Separation makes \(h_\sigma^*(x)=+\infty\) for \(x\notin C\);
therefore \(\operatorname{dom}\phi_\sigma\subset C\) and its
interior equals \(U\).

The Jacobian of
\((z,\sigma)\mapsto(D^{\MMSE}_\sigma(z),\sigma)\)
is invertible. The inverse is jointly smooth on
\(U\times(0,\infty)\), and differentiation of
\eqref{eq:phi_x_z} gives \(\nabla\phi_\sigma(x)=z-x\).
Gaussian convolution makes \(g_\sigma\) and
\(D^{\MMSE}_\sigma\) jointly real analytic for \(\sigma>0\).
The analytic inverse function theorem and \eqref{eq:phi_x_z}
give the same regularity for \(\phi_\sigma\) on
\(U\times(0,\infty)\).
\end{proof}

\paragraph{Learned Gradient-Step denoiser} Assume a gradient-step denoiser  $D_\sigma = \id - \nabla g_\sigma$ with contractive $\nabla g_\sigma$. Then $D_\sigma$ is the proximal operator of a weakly convex potential $\phi_\sigma$~\citep{hurault2022proximal}. In the following Proposition, we prove the different regularity properties of this potential. Note that parts of these proofs are already written in~\citep{hurault2022proximal, gribonval2020characterization}. We keep them here for completeness and clarity.

\begin{proposition}[Regularity of proximal gradient-step denoisers]
\label{prop:gs-proximal-regularity}
Suppose Assumption~\ref{ass:denoiser_gs} holds and, for some
\(\rho<1\), \(\operatorname{Lip}(\nabla g_\sigma)\leq\rho\)
uniformly for \(\sigma\in I_\eps\). Define
\(h_\sigma=\frac12\|\cdot\|^2-g_\sigma\) and
\(\phi_\sigma=h_\sigma^*-\frac12\|\cdot\|^2\), as in
\eqref{eq:phi_sigma}. Then \(D_\sigma=\id-\nabla g_\sigma\)
is a global \(\mathcal C^1\) diffeomorphism of \(\R^d\) and
\(D_\sigma=\operatorname{prox}_{\phi_\sigma}\).
The potential \(\phi_\sigma\) is finite, lower semicontinuous,
jointly \(\mathcal C^2\) and subanalytic in \((x,\sigma)\) on
\(\R^d\times I_\eps\). With \(z=D_\sigma^{-1}(x)\),
\begin{equation}
 \phi_\sigma(x)=g_\sigma(z)-\frac12\|z-x\|^2,
 \qquad
 \nabla\phi_\sigma(x)=z-x,
 \qquad
 \partial_\sigma\phi_\sigma(x)=\partial_\sigma g_\sigma(z).
 \label{eq:pot_phi}
\end{equation}
Moreover, $\nabla \phi_\sigma$ is globally Lipschitz with 
\[
 \operatorname{Lip}(\nabla\phi_\sigma)
 \leq\frac{\rho}{1-\rho}.
\]
Finally, uniformly for \(x\in\R^d\) and
\(\sigma,\sigma'\in I_\eps\),
\begin{equation}
 \begin{gathered}
 \phi_\sigma(x)\geq g_\sigma(x)\geq c_0\|x\|^2-C_0,\\
 \|D_\sigma^{-1}(x)\|
 \leq\frac{\|x\|+B}{1-\rho},
 \qquad B:=\sup_{\sigma\in I_\eps}\|D_\sigma(0)\|<\infty,\\
 |\phi_{\sigma'}(x)-\phi_\sigma(x)|
 \leq C|\sigma'-\sigma|(1+\|x\|^2).
 \end{gathered}
 \label{eq:prox-gs-global-variation}
\end{equation}
Moreover, on every compact \(\mathcal K\subset\R^d\), the quantities
\(\nabla\phi_\sigma\), \(\partial_\sigma\phi_\sigma\), and
\(\partial_\sigma\nabla\phi_\sigma\) are uniformly bounded on
\(\mathcal K\times I_\eps\).
\end{proposition}
\begin{proof}

For every \(x\), the equation \(D_\sigma(z)=x\) is
\(z=x+\nabla g_\sigma(z)\), whose right-hand side is a contraction.
Hence \(D_\sigma\) is bijective, and the inverse function theorem
gives a jointly \(\mathcal C^1\) inverse in \((x,\sigma)\). The bound on \(\nabla^2g_\sigma\) makes \(h_\sigma\) \((1-\rho)\)-strongly convex, 
thus \(h_\sigma^*\) is finite and \(\phi_\sigma\) is also finite and lower semi-continuous. For fixed $z$ and any $u\in\R^d$,
\begin{align*}
\frac12\|u-z\|^2+\phi_\sigma(u)
&=h_\sigma^*(u)-\langle z,u\rangle
  +\frac12\|z\|^2\\
&\geq - h_\sigma(z)+\frac12\|z\|^2
 =g_\sigma(z),
\end{align*}
where the inequality is Fenchel--Young:
$h_\sigma(z)+h_\sigma^*(u)\geq\langle z,u\rangle$.
Equality holds exactly when
$u = \nabla h_\sigma(z) = D_\sigma(z)$. 
Thus $D_\sigma(z)$ attains the lower bound and is the
unique minimizer. Consequently,
$D_\sigma(z)=\operatorname{prox}_{\phi_\sigma}(z)$.

Let $z=D_\sigma^{-1}(x)$. Since
$x=D_\sigma(z)=z-\nabla g_\sigma(z)=\nabla h_\sigma(z)$,
Fenchel equality gives
\[
h_\sigma^*(x)=\langle x,z\rangle-h_\sigma(z).
\]
Substituting the definitions of $h_\sigma$ and $\phi_\sigma$,
\[
\phi_\sigma(x)
=\langle x,z\rangle-\frac12\|z\|^2+g_\sigma(z)
 -\frac12\|x\|^2
=g_\sigma(z)-\frac12\|z-x\|^2.
\]
The chain rule yields
\[
\nabla_x\phi_\sigma(x)
=(J_xz)^\top\bigl(\nabla g_\sigma(z)-(z-x)\bigr)
 +(z-x)=z-x,
\]
because $\nabla g_\sigma(z)=z-x$. Likewise, differentiating
with respect to $\sigma$ while holding $x$ fixed gives
\[
\partial_\sigma\phi_\sigma(x)
=\partial_\sigma g_\sigma(z)
 +\bigl\langle\nabla g_\sigma(z)-(z-x),
                   \partial_\sigma z\bigr\rangle
=\partial_\sigma g_\sigma(z).
\]
As $z$ is jointly $\mathcal C^1$ and $g_\sigma$ is jointly
$\mathcal C^2$, these derivative identities also show that
$\phi_\sigma$ is jointly $\mathcal C^2$.

Now, for $x,y\in\R^d$, write
$u=D_\sigma^{-1}(x)$ and $v=D_\sigma^{-1}(y)$.
Since $D_\sigma=\Id-\nabla g_\sigma$,
\[
\|x-y\|
\geq \|u-v\|
     -\|\nabla g_\sigma(u)-\nabla g_\sigma(v)\|
\geq(1-\rho)\|u-v\|.
\]
Using $\nabla\phi_\sigma(x)=u-x=\nabla g_\sigma(u)$,
we obtain
\[
\|\nabla\phi_\sigma(x)-\nabla\phi_\sigma(y)\|
\leq\rho\|u-v\|
\leq\frac{\rho}{1-\rho}\|x-y\|.
\]

Since $g_\sigma$ is jointly $\mathcal C^2$ and subanalytic, the map $(z,\sigma)\mapsto D_\sigma(z)$,
is subanalytic. The graph of the jointly continuous inverse
$(x,\sigma)\mapsto D_\sigma^{-1}(x)$ is obtained by exchanging
the input and output coordinates in the graph of $D_\sigma$;
it is therefore subanalytic. Finally, the first identity in
\eqref{eq:pot_phi} expresses $\phi_\sigma(x)$ in terms of
$g_\sigma$ and this inverse. Since the inverse is locally
bounded, this composition is subanalytic.

Evaluating the conjugate defining \(\phi_\sigma(x)\) at \(z=x\)
gives \(\phi_\sigma(x)\geq g_\sigma(x)\), hence coercivity.

For \(z=D_\sigma^{-1}(x)\),
\[
 \|x-D_\sigma(0)\|
 =\|z-(\nabla g_\sigma(z)-\nabla g_\sigma(0))\|
 \geq(1-\rho)\|z\|,
\]
which proves the inverse bound. 

Assumption~\ref{ass:denoiser_gs},
\eqref{eq:pot_phi} and~\eqref{eq:prox-gs-global-variation} give
\(|\partial_\sigma\phi_\sigma(x)|\leq C(1+\|x\|^2)\).
Integrating this bound between \(\sigma\) and \(\sigma'\)
proves the variation estimate. 

Finally,
\(\partial_\sigma\nabla\phi_\sigma(x)
 =(I-\nabla^2g_\sigma(z))^{-1}
 \partial_\sigma\nabla g_\sigma(z)\).
The inverse bound sends \(\mathcal K\times I_\eps\) into a
compact set of \(z\)'s, so this formula and joint continuity
give all the asserted compact bounds.
\end{proof}

\paragraph{Common compact-set consequences.}
In the gradient-step case take \(U=\R^d\), and in the MMSE case
take \(U=\operatorname{int}C\). On every compact convex
\(\mathcal K\subset U\), joint \(\mathcal C^2\) regularity bounds
\(\partial_\sigma\phi_\sigma\),
\(\partial_\sigma\nabla\phi_\sigma\), and
\(\nabla\phi_\sigma\) uniformly on
\(\mathcal K\times I_\eps\).
For MMSE, joint continuity also maps every compact set of denoiser
inputs, uniformly over \(\sigma\in I_\eps\), into a compact subset
of \(U\).

\section{Convergence proofs}

\subsection{Proof of Theorem~\ref{thm:gd-pnp}}
\label{app:proof-gd-pnp}

\begin{proof}
Subtracting a lower bound from $f$, which does not change the iterates,
we may assume $f\geq0$.

We want to control $F_{\sigma_{k+1}}(x_{k+1})-F_{\sigma_k}(x_k)$.
Applying the descent lemma to the objective $F_{\sigma_k}$ used at
iteration $k$ gives, with $L=L_f+\lambda L_g$,
\begin{align}
 F_{\sigma_k}(x_{k+1})
 &\leq F_{\sigma_k}(x_k)
   +\langle \nabla F_{\sigma_k}(x_k),x_{k+1}-x_k\rangle
   +\frac{L}{2}\|x_{k+1}-x_k\|^2 \nonumber\\
 &=F_{\sigma_k}(x_k)
   -\gamma\left(1-\frac{L\gamma}{2}\right)\|\nabla F_{\sigma_k}(x_k)\|^2 \nonumber\\
 &\leq F_{\sigma_k}(x_k)-a\gamma\|\nabla F_{\sigma_k}(x_k)\|^2,
 \qquad a:=1-\frac{L\gamma}{2}>0.
 \label{eq:gd-current-descent}
\end{align}

We then prove that the iterates remain bounded. Recall the uniform coercivity from Assumption~\ref{ass:denoiser_gs}:
\begin{equation} \label{eq:unifco}
g_\sigma(x)\geq c_0\|x\|^2-C_0
\end{equation}
Set
$$V_k(x)=1+\lambda C_0+F_{\sigma_k}(x).$$ Then~\eqref{eq:unifco} gives
\begin{equation}
 V_k(x)\geq1+\lambda c_0\|x\|^2.
 \label{eq:boundedness-red-coercivity}
\end{equation}
Moreover,~\eqref{eq:unifco} also implies
\begin{equation}
 1+\|x\|^2\leq C\bigl(1+C_0+g_\sigma(x)\bigr).
 \label{eq:boundedness-energy-controls-norm}
\end{equation}
The mean-value theorem and the bound $|\partial_\sigma g_\sigma(x)| \leq M_g(1+\|x\|^2)$ from Assumption ~\ref{ass:denoiser_gs} give
\begin{equation}
 |g_{\sigma'}(x)-g_\sigma(x)|
 \leq M_g|\sigma'-\sigma|(1+\|x\|^2)
 \leq C|\sigma'-\sigma|\bigl(1+C_0+g_\sigma(x)\bigr).
 \label{eq:boundedness-potential-variation}
\end{equation}

By~\eqref{eq:boundedness-potential-variation},
\begin{equation}
 V_{k+1}(x)\leq(1+a_k)V_k(x),
 \qquad a_k=C|\sigma_{k+1}-\sigma_k|.
\end{equation}
Consequently,
\begin{equation}
 V_{k+1}(x_{k+1})
 \leq(1+a_k)V_k(x_{k+1})
 \leq(1+a_k)V_k(x_k).
\end{equation}
The monotone noise schedule has finite total variation, hence
$\prod_k(1+a_k)<\infty$ and $\sup_kV_k(x_k)<\infty$.  The uniform
coercivity~\eqref{eq:boundedness-red-coercivity} then yields
$\sup_k\|x_k\|<\infty$. 

We can now control the change of objective between two successive noise levels.
Let $\cK$ be a compact set containing all the iterates. Since \((x,\sigma)\mapsto g_\sigma(x)\) is jointly \( \mathcal{C}^2\), continuity of \(\partial_\sigma \nabla g_\sigma(x)\) and compactness of \(\cK\times I_\epsilon\) imply that \(\sigma\mapsto \nabla g_\sigma(x)\) is Lipschitz continuous uniformly for \(x\in\cK\), with constant $L^\sigma_{g,\cK}$.

Since $f$ does not
depend on $\sigma$, the fundamental theorem of calculus and Assumption~\ref{ass:denoiser_gs} give, for every $x\in\cK$ and
$\sigma,\sigma'\in I_\epsilon$,
\begin{align}
 |F_\sigma(x)-F_{\sigma'}(x)|
 &= \lambda |g_\sigma(x)-g_{\sigma'}(x)| \nonumber\\
 &= \lambda\left|
    \int_{\sigma'}^\sigma \partial_s g_s(x)\,ds
    \right| \nonumber\\
 &\le \lambda M_g (1+\|x\|^2)|\sigma-\sigma'| \nonumber\\
 &\le C_{\cK}|\sigma-\sigma'|,
 \label{eq:potential-variation}
\end{align}
where
\[
 C_{\cK}
 := \lambda M_g\left(1+\sup_{x\in\cK}\|x\|^2\right)
 <\infty.
\]
Combining \eqref{eq:gd-current-descent} and
\eqref{eq:potential-variation} at the point $x_{k+1}$ yields
\begin{equation}
 F_{\sigma_{k+1}}(x_{k+1})
 \leq F_{\sigma_k}(x_k)-a\gamma\|\nabla F_{\sigma_k}(x_k)\|^2
      + C_{\cK}|\sigma_{k+1}-\sigma_k|.
 \label{eq:gd-quasi-descent}
\end{equation}
Define
\begin{equation}
 r_k:= C_{\cK}\sum_{j=k}^\infty|\sigma_{j+1}-\sigma_j|.
\end{equation}
It is finite and tends to zero by the monotone convergence of $(\sigma_k)$
in Assumption~\ref{ass:noise-schedule}. Since
$r_k-r_{k+1}= C_{\cK}|\sigma_{k+1}-\sigma_k|$,
\eqref{eq:gd-quasi-descent} implies
\begin{equation}
  F_{\sigma_{k+1}}(x_{k+1})+r_{k+1}
 \leq F_{\sigma_k}(x_k)+r_k
      -a\gamma\|\nabla F_{\sigma_k}(x_k)\|^2.
 \label{eq:gd-corrected-descent}
\end{equation}
Thus $\big(  F_{\sigma_{k}}(x_{k})+r_{k}\big)$ is a decreasing sequence, and bounded below, it is therefore convergent.  Because
$r_k\to 0$, the sequence $ \big(  F_{\sigma_k}(x_k) \big)$ also converges,
and summing \eqref{eq:gd-corrected-descent} gives
\begin{equation}
 \sum_{k=0}^\infty\|\nabla F_{\sigma_k}(x_k)\|^2<\infty.
 \label{eq:gd-current-energy}
\end{equation}
Let $F_\star:=\inf_{\sigma\in I_\eps,\,x\in\cK}F_\sigma(x)>-\infty$.
More precisely, for every $K\geq0$,
\begin{equation}
 \min_{0\leq k\leq K}\|\nabla F_{\sigma_k}(x_k)\|^2
 \leq
 \frac{ F_{\sigma_0}(x_0)- F_\star+
  C_{\cK}\sum_{k=0}^K|\sigma_{k+1}-\sigma_k|}
 {a\gamma(K+1)}.
 \label{eq:gd-current-rate-proof}
\end{equation}
The numerator is uniformly bounded in $K$, proving the rate stated in
\eqref{eq:gd-pnp-limits}.

Since the stepsize is a fixed positive constant,
\eqref{eq:gd-current-energy} directly implies
\begin{equation}
 \|\nabla F_{\sigma_k}(x_k)\|\longrightarrow0.
 \label{eq:gd-current-gradient-limit}
\end{equation}

Lipschitz continuity with respect to $\sigma$ on $\cK$ also gives
\begin{equation}
 \|\nabla F_{\sigma_k}(x_k)-\nabla F_\eps(x_k)\|
 \leq\lambda L^\sigma_{g,\cK} |\sigma_k-\eps|\longrightarrow0.
 \label{eq:gd-terminal-gradient-transfer}
\end{equation}
Combining this with \eqref{eq:gd-current-gradient-limit} proves
$\|\nabla F_\eps(x_k)\|\to0$.  If $x_{k_j}\to\bar x$ is any convergent
subsequence, continuity of $\nabla F_\eps$ gives
$\nabla F_\eps(\bar x)=0$, so every cluster point is critical.

Taking $\sigma'=\eps$ in \eqref{eq:potential-variation} gives
\[
 | F_{\sigma_k}(x_k)- F_\eps(x_k)|
 \leq C_{\cK}|\sigma_k-\eps|\longrightarrow0.
\]
We already proved convergence of the first term; hence
$F_\eps(x_k)$ converges. 

Moreover, since
\[
  x_{k+1}-x_k=-\gamma\nabla F_{\sigma_k}(x_k)
\]
and \eqref{eq:gd-current-gradient-limit}
yields $\|x_{k+1}-x_k\|\to0$.
Since $(x_k)$ is bounded and every cluster point is critical for
$F_\epsilon$, we also obtain
\[
  \operatorname{dist}(x_k,\operatorname{crit}F_\epsilon)\to0.
\]

\paragraph{Rate for the terminal gradient.}
Suppose that $\sum_k|\sigma_k-\eps|^2<\infty$, as in the rate
claim following Theorem~\ref{thm:gd-pnp}. Let
\begin{equation}
 b_k:=\nabla F_{\sigma_k}(x_k)-\nabla F_\eps(x_k)
 =\lambda\bigl(\nabla g_{\sigma_k}(x_k)
                    -\nabla g_\eps(x_k)\bigr),
 \label{eq:gd-envelopes}
\end{equation}
so that $x_{k+1}=x_k-\gamma(\nabla F_\eps(x_k)+b_k)$ and
\begin{equation}
 \|b_k\|\leq\lambda L^{\sigma}_{g, \cK}|\sigma_k-\eps|.
 \label{eq:gd-bias-bound}
\end{equation}
Let $\delta:=2-L\gamma>0$, $c:=\delta/4$, and
$C:=1+1/\delta$.  Using the descent lemma on $F_\eps$ gives
\begin{align}
 F_\eps(x_{k+1})
 \leq{}&F_\eps(x_k)
 -\gamma\left(1-\frac{L\gamma}{2}\right)
       \|\nabla F_\eps(x_k)\|^2\nonumber\\
 &-\gamma(1-L\gamma)
       \langle\nabla F_\eps(x_k),b_k\rangle
 +\frac{L\gamma^2}{2}\|b_k\|^2.
 \label{eq:gd-terminal-expansion}
\end{align}
Since $1-L\gamma/2=\delta/2$,
$|1-L\gamma|\leq1$, and
\[
 \|\nabla F_\eps(x_k)\|\,\|b_k\|
 \leq\frac{\delta}{4}\|\nabla F_\eps(x_k)\|^2+
      \frac{1}{\delta}\|b_k\|^2,
\]
we obtain
\begin{equation}
 F_\eps(x_{k+1})
 \leq F_\eps(x_k)
 -c\gamma\|\nabla F_\eps(x_k)\|^2
 +C\gamma\|b_k\|^2.
 \label{eq:gd-terminal-descent}
\end{equation}
The last terms are summable by
\eqref{eq:gd-bias-bound}.  Since $F_\eps$ is bounded below, summation gives
\begin{equation}
 \min_{0\leq k\leq K}\|\nabla F_\eps(x_k)\|^2
 \leq
 \frac{F_\eps(x_0)-\inf F_\eps+
 C \gamma \lambda^2(L^{\sigma}_{g, \cK})^2
   \sum_{k=0}^K|\sigma_k-\eps|^2}
 {c \gamma (K+1)}.
 \label{eq:gd-terminal-rate-proof}
\end{equation}

\paragraph{Point convergence.}
Finally, for convergence of the sequence, we apply the inexact Kurdyka--\L{}ojasiewicz theorem of
\citet[Theorem~10]{ochs2019inexact}. 

We need to verify \citet[Assumption H]{ochs2019inexact}. Under the
additional decay assumption of Theorem~\ref{thm:gd-pnp},
\eqref{eq:gd-envelopes} gives
\begin{equation}
 e_k:=\|\nabla F_{\sigma_k}(x_k)-\nabla F_\eps(x_k)\|=\mathcal O(k^{-\alpha}),\qquad
 \sum_ke_k<\infty,\qquad \sum_ke_k^2<\infty.
 \label{eq:gd-kl-errors}
\end{equation}
Let $m:=1/\gamma-L/2>0$. Using the identity
\begin{equation} \label{eq:nablaFeps}
 \nabla F_\eps(x_k)=-\frac{x_{k+1}-x_k}{\gamma}-\left( \nabla F_{\sigma_k}(x_k)-\nabla F_\eps(x_k) \right)
\end{equation}
in the descent lemma on $F_\eps$ gives
\begin{align*}
 F_\eps(x_{k+1})
 &\leq F_\eps(x_k)
 -\left(\frac{1}{\gamma}-\frac{L}{2}\right)\|x_{k+1} - x_k\|^2+
 e_k \,\|x_{k+1} - x_k\|
 \end{align*}
 Young's inequality
 \begin{align}
  e_k \,\|x_{k+1} - x_k\| \leq \frac{m}{2} \|x_{k+1} - x_k\|^2 + \frac{1}{2m}e_k^2
 \end{align}
 gives
\begin{align}
 F_\eps(x_{k+1}) \leq F_\eps(x_k)-\frac{m}{2}\|x_{k+1} - x_k\|^2+
      \frac{1}{2m}e_k^2.
 \label{eq:gd-kl-pseudo-descent}
\end{align}
The last term is an additional error term compared to the sufficient decrease condition from Assumption~H1 of
\citet{ochs2019inexact}. We can however absorb it in a different Lyapunov function. Choose an even
integer
\begin{equation*}
 \theta >\frac{2\alpha-1}{2\alpha-2},
\end{equation*}
and define
\begin{equation*}
 t_k:=\left(\frac{ \theta}{2m}\sum_{j=k}^{\infty}e_j^2\right)^{1/ \theta},
 \qquad
 \Phi(x,t):=F_\eps(x)+\frac{t^ \theta}{ \theta}.
\end{equation*}
Then \eqref{eq:gd-kl-pseudo-descent} rewrites on $\Phi(x,t)$ as
\begin{equation*}
 \Phi(x_{k+1},t_{k+1})+\frac m2\|x_{k+1}-x_k\|^2
 \leq \Phi(x_k,t_k),
\end{equation*}
which is the sufficient-decrease condition in Assumption~H1 of
\citet{ochs2019inexact}.  
Now, Lipschitz continuity of $\nabla F_\eps$ gives 
\begin{align*}
 \|\nabla F_\eps(x_{k+1})\|
 &\leq\|\nabla F_\eps(x_{k+1})-\nabla F_\eps(x_k)\|
      +\|\nabla F_\eps(x_k)\|\nonumber\\
 &\leq  L \|x_{k+1} - x_k\| + \|\nabla F_\eps(x_k)\|
\end{align*}
From~\eqref{eq:nablaFeps},
\begin{align*} 
 \| \nabla F_\eps(x_k)  \| &=  \left\|\frac{x_{k+1}-x_k}{\gamma} + \left( \nabla F_{\sigma_k}(x_k)-\nabla F_\eps(x_k) \right)  \right\| \\
 &\leq \frac{1}{\gamma} \|x_{k+1} - x_k\| + e_k
\end{align*}
and thus 
\begin{align*}
 \|\nabla F_\eps(x_{k+1})\|
 &\leq  \left( L + \frac{1}{\gamma}  \right) \|x_{k+1} - x_k\| + e_k.
\end{align*}
And in terms of $\Phi(x,t)$
\begin{equation*}
 \|\nabla\Phi(x_{k+1},t_{k+1})\|
 \leq \left(L+\frac{1}{\gamma}\right)
       \|x_{k+1}-x_k\|+e_k+t_{k+1}^{ \theta-1}.
\end{equation*}
Since $e_k=\mathcal O(k^{-\alpha})$,
\begin{equation*}
 t_k^{ \theta-1}
 =\mathcal O\!\left(k^{-(2\alpha-1)(\theta-1)/\theta}\right).
\end{equation*}
The choice of $\theta$ therefore makes both $(e_k)$ and
$(t_k^{\theta-1})$ summable, and we thus verify the relative-error condition from \citet[Assumption H2]{ochs2019inexact}. Moreover, since $(x_k)$ is
bounded and $t_k\to0$, the sequence $(x_k,t_k)$ is bounded and admits a
convergent subsequence. Continuity of $\Phi$ makes this subsequence
$\Phi$-attentively convergent, which verifies~(H3). The distance and
parameter conditions follow by taking the distance
$\|x_{k+1}-x_k\|$ and constant positive coefficients.

Finally, since $F_\eps$ is subanalytic (as $f$ and $g_\sigma$ are by assumption) and $t\mapsto t^\theta$ is
polynomial, $\Phi$ is subanalytic and hence KL
\citep{bolte2007lojasiewicz}. Since $\theta$ is even and $F_\eps$ is
bounded below, $\Phi$ is also bounded below. Therefore,
\citet[Theorem~10]{ochs2019inexact} yields
\[
 \sum_k\|x_{k+1}-x_k\|<\infty,\qquad
 x_k\longrightarrow x_\star,\qquad
 \nabla F_\eps(x_\star)=0.
\]
\end{proof}

\subsection{Proof of Theorem~\ref{thm:snore}}
\label{app:proof-sgpnp}

\begin{proof}
In this proof, $F_\sigma=f+\lambda\tilde g_\sigma$ denotes the SNORE
objective, with $\tilde g_\sigma$ defined in
\eqref{eq:snore-averaged-potential}. We first check that its Gaussian
average is well defined. Assumption~\ref{ass:denoiser_gs} and continuity
on $I_\eps$ give a constant $C$ such that, uniformly in
$\sigma\in I_\eps$,
\[
 \|\nabla g_\sigma(y)\|\leq C+L_g\|y\|,
 \qquad |g_\sigma(y)|\leq C(1+\|y\|^2).
\]
Thus Gaussian integration and differentiation in $x$ give
\begin{equation}
 \nabla\tilde g_\sigma(x)
 =\E_Z\nabla g_\sigma(x+\sigma Z)
 =\E_Z\bigl[x-D_\sigma(x+\sigma Z)\bigr].
 \label{eq:snore-mean-gradient}
\end{equation}
The uniform $L_g$-Lipschitz continuity of $\nabla g_\sigma$ implies
that $\nabla\tilde g_\sigma$ is also globally $L_g$-Lipschitz. Hence
$\nabla F_\sigma$ is globally $L=L_f+\lambda L_g$-Lipschitz for every
$\sigma\in I_\eps$.

Let $\cF_k$ contain the information available before drawing
$Z_{k+1}$, and set
\[
 \xi_{k+1}:=\lambda\bigl(
  \E_ZD_{\sigma_k}(x_k+\sigma_kZ)
  -D_{\sigma_k}(x_k+\sigma_kZ_{k+1})\bigr).
\]
By \eqref{eq:snore-mean-gradient}, the SNORE update is
\begin{equation}
 x_{k+1}=x_k-\gamma_k(\nabla F_{\sigma_k}(x_k)+\xi_{k+1}),
 \qquad \E[\xi_{k+1}\mid\cF_k]=0.
 \label{eq:snore-moving-recursion}
\end{equation}
Moreover, $D_\sigma=\id-\nabla g_\sigma$ is
$(1+L_g)$-Lipschitz. If $Z'$ and $Z$ are independent, we get, conditionally on $\cF_k$,
\begin{align}
 \E[\|\xi_{k+1}\|^2\mid\cF_k]
 &=\frac{\lambda^2}{2}\E_{Z,Z'}
   \|D_{\sigma_k}(x_k+\sigma_kZ)
     -D_{\sigma_k}(x_k+\sigma_kZ')\|^2\nonumber\\
&\leq \frac{\lambda^2}{2} \sigma_k^2 (1+L_g)^2 \E_{Z,Z'} \|Z - Z' \|^2
\nonumber \\ 
 &\leq \lambda^2\sigma_0^2(1+L_g)^2d := C_\xi.
 \label{eq:snore-variance-bound}
\end{align}

We next control the objective used at iteration $k$, as in the RED--GD
proof. The descent lemma, \eqref{eq:snore-moving-recursion}, and
conditional expectation yield
\begin{equation}
 \E[F_{\sigma_k}(x_{k+1})\mid\cF_k]
 \leq F_{\sigma_k}(x_k)-a\gamma_k\|\nabla F_{\sigma_k}(x_k)\|^2
       +\frac{L}{2} C_\xi \gamma_k^2,
 \qquad a:=1-\frac{L\bar\gamma}{2}>0.
 \label{eq:snore-current-descent}
\end{equation}

We now prove that the iterates remain bounded. After subtracting a
lower bound from $f$, we may assume $f\geq0$. Uniform coercivity of
$g_\sigma$ (Assumption~\ref{ass:denoiser_gs}) passes to its Gaussian average:
\begin{equation}
 \tilde g_\sigma(x)
 \geq c_0\E\|x+\sigma Z\|^2-C_0
 \geq c_0\|x\|^2-C_0.
 \label{eq:snore-coercivity}
\end{equation}
Let
\[
 V_k(x):=1+\lambda C_0+F_{\sigma_k}(x).
\]
Then $V_k(x)\geq1+\lambda c_0\|x\|^2$, and in particular
$1+\|x\|^2\leq C V_k(x)$ for a constant independent of $k$.
Using the chain rule, 
\begin{align}
 |\partial_\sigma\tilde g_\sigma(x)|
 &=\left|\E_Z\left[\partial_\sigma g_\sigma(x+\sigma Z)
       +\langle\nabla g_\sigma(x+\sigma Z),Z\rangle\right]\right|
\end{align}
Assumption~\ref{ass:denoiser_gs} then gives:
\begin{align}
 |\partial_\sigma\tilde g_\sigma(x)|
 \leq C(1+\|x\|^2).
  \label{eq:snore-potential-derivative}
\end{align}
Consequently, with the mean-value theorem, we get, for all $x$ and
$\sigma,\sigma'\in I_\eps$,
\begin{equation}
 |F_{\sigma'}(x)-F_\sigma(x)|
 \leq C|\sigma'-\sigma|(1+\|x\|^2).
 \label{eq:pot_var}
\end{equation}
Writing $a_k=C|\sigma_{k+1}-\sigma_k|$ with a larger constant if
necessary, we obtain $V_{k+1}(x)\leq(1+a_k)V_k(x)$.
Combining this inequality with \eqref{eq:snore-current-descent} gives
\begin{equation}
 \E[V_{k+1}(x_{k+1})\mid\cF_k]
 \leq(1+a_k)V_k(x_k)-a\gamma_k\|\nabla F_{\sigma_k}(x_k)\|^2+b_k,
 \qquad b_k:=\frac{L}{2}C_\xi(1+a_k)\gamma_k^2.
 \label{eq:snore-coercive-descent}
\end{equation}
The monotone schedule has finite total variation, so
$\sum_k a_k<\infty$, and the stepsize condition gives
$\sum_k b_k<\infty$. The Robbins--Siegmund
almost-supermartingale theorem
\citep[Theorem~1]{robbins1971convergence}, applied to
\eqref{eq:snore-coercive-descent}, shows that $V_k(x_k)$
converges almost surely to a finite random variable and, since $a>0$,
\begin{equation}
 \sum_k\gamma_k\|\nabla F_{\sigma_k}(x_k)\|^2<\infty
 \quad\text{almost surely}.
 \label{eq:snore-current-energy}
\end{equation}
The coercivity bound $1+\|x\|^2\leq C V_k(x)$  also gives
$\sup_k\|x_k\|<\infty$ almost surely.

Taking expectations in \eqref{eq:snore-coercive-descent} also gives
$\sup_k\E[V_k(x_k)]<\infty$. Summing the same inequality through
$K$ yields
\begin{equation}
 a\sum_{k=0}^K\gamma_k\E[\|\nabla F_{\sigma_k}(x_k)\|^2]
 \leq \E[V_0(x_0)]+ \sup_k\E[V_k(x_k)]\sum_{k=0}^K a_k
                    +\sum_{k=0}^K b_k.
 \label{eq:snore-rate-proof}
\end{equation}
The right-hand side is bounded independently of $K$; division by
$\sum_{k=0}^K\gamma_k$ proves \eqref{eq:sg-current-rate}.

Now that we have bounded iterates, we can use local regularity in the noise level. Fix a sample path
on which the above conclusions hold, and let $\cK$ be a compact convex
ball containing all its iterates. 

Recall the parameter-derivative identity
$$
\partial_\sigma\tilde g_\sigma(x)
=\mathbb E\!\left[\partial_\sigma g_\sigma(x+\sigma Z)
+\langle\nabla g_\sigma(x+\sigma Z),Z\rangle\right].
$$
Differentiating with respect to $x$ and applying the Gaussian integration by
parts
$\nabla_x\E_Z[\phi(x+\sigma Z)]
=\sigma^{-1}\E_Z[Z\phi(x+\sigma Z)]$, we get
\begin{equation}
 \partial_\sigma\nabla\tilde g_\sigma(x)
 =\frac1\sigma\E_Z\!\left[Z\,\partial_\sigma
             g_\sigma(x+\sigma Z)\right]
  +\E_Z\!\left[\nabla^2g_\sigma(x+\sigma Z)Z\right].
 \label{eq:snore-noise-gradient-derivative}
\end{equation}
The first expectation is bounded on $\cK\times I_\eps$ by the
quadratic growth of $\partial_\sigma g_\sigma$ (Assumption~\ref{ass:denoiser_gs}), and the second by
$L_g\E\|Z\|$. Thus there is a finite constant
$L^\sigma_{\tilde g,\cK}$ such that
\begin{equation}
 \sup_{x\in\cK}
 \|\nabla F_\sigma(x)-\nabla F_{\sigma'}(x)\|
 \leq\lambda L^\sigma_{\tilde g,\cK}|\sigma-\sigma'|,
 \qquad \sigma,\sigma'\in I_\eps.
 \label{eq:snore-gradient-convergence}
\end{equation}
In particular, $
 \|\nabla F_{\sigma_k}(x_k)-\nabla F_{\epsilon}(x_k)\| \to 0$.

Consider the martingale $M_n:=\sum_{j=0}^{n-1}\gamma_j\xi_{j+1}$. Since $\sum\gamma_k^2<\infty$ and \eqref{eq:snore-variance-bound}, it converges almost
surely.  
On a sample path where this holds, let
\[
 A_n:=\sum_{j=n}^\infty\gamma_j\|\,\nabla F_{\sigma_j} (x_j)\|^2,
 \qquad R_n:=\sup_{m\geq n}\|M_m-M_n\|.
\]
Both tend to zero. For any $T>0$, let $m(n)$ be the first index such
that $\sum_{j=n}^{m(n)-1}\gamma_j\geq T$. For $n\leq l\leq m(n)$, summing
the recursion~\eqref{eq:snore-moving-recursion} gives
\[
 x_l-x_n
 =-\sum_{j=n}^{l-1}\gamma_j \nabla F_{\sigma_j}(x_j) -(M_l-M_n).
\]
Hence, by Cauchy--Schwarz,
\begin{align*}
 \|x_l-x_n\|
 &\leq \sum_{j=n}^{l-1}\gamma_j\|\nabla F_{\sigma_j}(x_j)\|+R_n\\
 &\leq
 \left(\sum_{j=n}^{l-1}\gamma_j\right)^{1/2}
 \left(\sum_{j=n}^{l-1}\gamma_j\|\nabla F_{\sigma_j}(x_j)\|^2\right)^{1/2}
 +R_n.
\end{align*}
By the minimality of $m(n)$ and $\gamma_j\leq\bar\gamma$,
the first sum is at most $T+\bar\gamma$; the second is at
most $A_n$. The bound is uniform in $l$, so
\begin{equation}
 \max_{n\leq l\leq m(n)}\|x_l-x_n\|
 \leq\sqrt{(T+\bar\gamma)A_n}+R_n\longrightarrow0.
 \label{eq:snore-window}
\end{equation}
By the global $L$-Lipschitz continuity in $x$ and
\eqref{eq:snore-gradient-convergence},
\[
 \max_{n\leq l\leq m(n)}\|\nabla F_{\sigma_l}(x_l)-\nabla F_{\sigma_n}(x_n)\|
 \leq L\max_{n\leq l\leq m(n)}\|x_l-x_n\|
 +2\lambda L^\sigma_{\tilde g,\cK}
    \sup_{j\geq n}|\sigma_j-\eps|\longrightarrow0.
\]
Suppose that $\|\nabla F_{\sigma_n}(x_n)\|$ does not tend to zero.
Then there are $\eta>0$ and arbitrarily large indices $n$ such that
$\|\nabla F_{\sigma_n}(x_n)\|\geq\eta$. The preceding uniform
gradient estimate shows that, for every sufficiently large such $n$,
\[
 \max_{n\leq l\leq m(n)}
 \|\nabla F_{\sigma_l}(x_l)
   -\nabla F_{\sigma_n}(x_n)\|
 \leq\frac{\eta}{2}.
\]
The triangle inequality therefore gives
$\|\nabla F_{\sigma_l}(x_l)\|\geq\eta/2$ for every
$n\leq l\leq m(n)$. By the definition of $m(n)$,
\[
 A_n
 \geq\sum_{j=n}^{m(n)-1}
       \gamma_j\|\nabla F_{\sigma_j}(x_j)\|^2
 \geq\frac{\eta^2}{4}
       \sum_{j=n}^{m(n)-1}\gamma_j
 \geq\frac{T\eta^2}{4},
\]
contradicting $A_n\to0$. Thus
$\|\nabla F_{\sigma_k}(x_k)\|\to0$, and
\eqref{eq:snore-gradient-convergence} gives
$\|\nabla F_\eps(x_k)\|\to0$.

Finally, we get from~\eqref{eq:pot_var} $|F_{\sigma_k}(x_k)-F_\eps(x_k)|\to0$. Since $V_k(x_k)$
converges, so does $F_\eps(x_k)$. Moreover,
$\gamma_k\to0$, $\nabla F_{\sigma_k}(x_k)$ is bounded, and convergence of $M_n$ implies
$\gamma_k\xi_{k+1}\to0$. Hence
$\|x_{k+1}-x_k\|\to0$. Continuity of
$\nabla F_\eps$ makes every cluster point critical; boundedness and
compactness of $\cK$ yield
$\dist(x_k,\crit F_\eps)\to0$.
\end{proof}

\subsection{Annealed ERED}
\label{app:ered}

We consider a probability
distribution $\pi$ supported on a subgroup of orthogonal transformations. The equivariant RED iterates (ERED)~\citep{terris2024} write:
\begin{equation}
\text{(ERED)}\qquad
x_{k+1}=x_k-\gamma_k\Bigl[\nabla f(x_k)
+\lambda\bigl(x_k-Q_{k+1}^\top
D_{\sigma_k}(Q_{k+1}x_k)\bigr)\Bigr],
\qquad Q_{k+1}\overset{\mathrm{i.i.d.}}{\sim}\pi.
\label{eq:ered}
\end{equation}
For the conservative denoiser~\eqref{eq:gs-denoiser}, set
\begin{equation}
 \bar g_\sigma(x)=\E_Q[g_\sigma(Qx)],\qquad
 F_\sigma(x)=f(x)+\lambda\bar g_\sigma(x).
 \label{eq:ered-potential}
\end{equation}
The compactness of $O(\mathbb R^d)$ and Assumption~\ref{ass:denoiser_gs} justify
differentiation under the expectation and give
\begin{equation}
 \nabla\bar g_\sigma(x)
 =\E_Q\bigl[Q^\top\nabla g_\sigma(Qx)\bigr].
 \label{eq:ered-mean-gradient}
\end{equation}

The following theorem gives for ERED~\eqref{eq:ered} the
same convergence guarantees as Theorem~\ref{thm:snore} gives for SNORE.
\begin{theorem}[Convergence of ERED~\eqref{eq:ered}]
\label{thm:ered}
Suppose Assumptions~\ref{ass:noise-schedule},
\ref{ass:prox-lipschitz-x}, and~\ref{ass:denoiser_gs} hold, and the
transformations are sampled as above. For a deterministic starting
point $x_0$, suppose the stepsizes satisfy
\begin{equation}
 0<\gamma_k\leq\bar\gamma,\qquad
 (L_f+\lambda L_g)\bar\gamma<2,\qquad
 \sum_k\gamma_k=\infty, \qquad
 \sum_k\gamma_k^2<\infty.
 \label{eq:ered-descent-conditions}
\end{equation}
Then the iterates are almost surely bounded, $ F_\epsilon(x_k)$
converges almost surely, and
\begin{equation}
 \sum_k\gamma_k\|\nabla F_{\sigma_k}(x_k)\|^2<\infty,
 \qquad\|\nabla F_\epsilon(x_k)\|\to0,
 \qquad\operatorname{dist}(x_k,\operatorname{crit} F_\epsilon)
 \to0\quad\text{almost surely}.
 \label{eq:ered-limits}
\end{equation}
Moreover, $\|x_{k+1}-x_k\|\to0$ almost surely and
\begin{equation}
 \min_{0\leq k\leq K}
 \E\|\nabla F_{\sigma_k}(x_k)\|^2
 =\mathcal O\!\left((\textstyle\sum_{k=0}^{K}\gamma_k)^{-1}\right).
 \label{eq:ered-current-rate}
\end{equation}
\end{theorem}

\begin{proof} The proof follows the same arguments as the proof of SNORE.
Let $L:=L_f+\lambda L_g$. Orthogonality and
\eqref{eq:ered-mean-gradient} give, for every $x,y\in\R^d$,
\[
 \|\nabla\bar g_\sigma(x)-\nabla\bar g_\sigma(y)\|
 \leq\E_Q\|\nabla g_\sigma(Qx)-\nabla g_\sigma(Qy)\|
 \leq L_g\|x-y\|.
\]
Thus $\nabla F_\sigma$ is globally $L$-Lipschitz, uniformly in
$\sigma$. Let $\cF_k$ contain the information available before
drawing $Q_{k+1}$, and set
\[
 \xi_{k+1}:=\lambda\bigl(
 Q_{k+1}^\top\nabla g_{\sigma_k}(Q_{k+1}x_k)
 -\nabla\bar g_{\sigma_k}(x_k)\bigr).
\]
Then ERED~\eqref{eq:ered} becomes
\begin{equation}
 x_{k+1}=x_k-\gamma_k(\nabla F_{\sigma_k}(x_k)+\xi_{k+1}),
 \qquad \E[\xi_{k+1}\mid\cF_k]=0.
 \label{eq:ered-moving-recursion}
\end{equation}

The variance of this stochastic gradient depends on $x_k$.
Continuity on $I_\eps$ gives
$C_g:=\sup_{\sigma\in I_\eps}\|\nabla g_\sigma(0)\|<\infty$.
Since $Q$ is orthogonal,
\[
 \|Q^\top\nabla g_\sigma(Qx)\|
 \leq C_g+L_g\|x\|,
\]
and conditional centering yields
\begin{equation}
 \E[\|\xi_{k+1}\|^2\mid\cF_k]
 \leq 2\lambda^2(C_g^2+L_g^2\|x_k\|^2).
 \label{eq:ered-variance-growth}
\end{equation}
The descent lemma and $\E[\xi_{k+1}\mid\cF_k]=0$ now give
\begin{equation}
 \E[ F_{\sigma_k}(x_{k+1})\mid\cF_k]
 \leq F_{\sigma_k}(x_k)
       -a\gamma_k\|\nabla F_{\sigma_k}(x_k)\|^2
       +\frac L2\gamma_k^2
          \E[\|\xi_{k+1}\|^2\mid\cF_k],
 \qquad a:=1-\frac{L\bar\gamma}{2}>0.
 \label{eq:ered-current-descent}
\end{equation}

We prove boundedness of the iterates. Subtracting a
lower bound from $f$, we may assume $f\geq0$. Orthogonality and
Assumption~\ref{ass:denoiser_gs} imply
\begin{equation}
 \bar g_\sigma(x)\geq c_0\|x\|^2-C_0,
 \qquad
 |\partial_\sigma\bar g_\sigma(x)|
 =|\E_Q\partial_\sigma g_\sigma(Qx)|
 \leq M_g(1+\|x\|^2).
 \label{eq:ered-coercivity-variation}
\end{equation}
Set $V_k(x):=1+\lambda C_0+ F_{\sigma_k}(x)$. Then
$V_k(x)\geq1+\lambda c_0\|x\|^2$, and hence
$1+\|x\|^2\leq C V_k(x)$ uniformly in $k$. The mean-value theorem
therefore gives
\begin{equation}
 | F_{\sigma'}(x)- F_\sigma(x)|
 \leq\lambda M_g|\sigma'-\sigma|(1+\|x\|^2),
 \qquad
 V_{k+1}(x)\leq(1+a_k)V_k(x),
 \label{eq:ered-potential-variation}
\end{equation}
where $a_k=C|\sigma_{k+1}-\sigma_k|$. By
\eqref{eq:ered-variance-growth} and the coercivity bound, there is a
deterministic $C_\xi$ such that
$\E[\|\xi_{k+1}\|^2\mid\cF_k]\leq C_\xi V_k(x_k)$.
Write $b_k=(L/2)C_\xi\gamma_k^2$ and
$\alpha_k=(1+a_k)(1+b_k)-1$. Combining
\eqref{eq:ered-current-descent} with
\eqref{eq:ered-potential-variation} yields
\begin{equation}
 \E[V_{k+1}(x_{k+1})\mid\cF_k]
 \leq(1+\alpha_k)V_k(x_k)-a\gamma_k\|\nabla F_{\sigma_k}(x_k)\|^2.
 \label{eq:ered-coercive-descent}
\end{equation}
The monotone schedule and the stepsize assumptions give
$\sum_k\alpha_k<\infty$. The Robbins--Siegmund
almost-supermartingale theorem applied to
\eqref{eq:ered-coercive-descent} shows that $V_k(x_k)$ converges
almost surely to a finite random variable and
\begin{equation}
 \sum_k\gamma_k\|\nabla F_{\sigma_k}(x_k)\|^2<\infty
 \quad\text{almost surely}.
 \label{eq:ered-current-energy}
\end{equation}
Because $1+\|x_k\|^2\leq C V_k(x_k)$, the iterates are almost surely
bounded.

Taking expectations in \eqref{eq:ered-coercive-descent} gives
$\sup_k\E V_k(x_k)<\infty$. Summing the same inequality through $K$
then yields
\[
 a\sum_{k=0}^K\gamma_k\E\big[\|\nabla F_{\sigma_k}(x_k)\|^2]\big]
 \leq \E \big[V_0(x_0)\big]
      +\sup_j\E \big[V_j(x_j)\big]\sum_{k=0}^K\alpha_k.
\]
The right-hand side is bounded independently of $K$, so division by
$\sum_{k=0}^K\gamma_k$ proves \eqref{eq:ered-current-rate}.

Fix a sample path on which the iterates are bounded and let $\cK$ be
a compact ball centered at the origin and containing them. Since
$Q\cK=\cK$ for $Q\in O(d)$
and $g_\sigma$ is jointly $C^2$,
\[
 L^\sigma_{g,\cK}
 :=\sup_{\substack{y\in\cK\\\sigma\in I_\eps}}
       \|\partial_\sigma\nabla g_\sigma(y)\|<\infty.
\]
Using \eqref{eq:ered-mean-gradient} and the mean-value theorem, we get
\begin{equation}
 \sup_{x\in\cK}
 \|\nabla F_\sigma(x)-\nabla F_{\sigma'}(x)\|
 \leq\lambda L^\sigma_{g,\cK}|\sigma-\sigma'|.
 \label{eq:ered-noise-regularity}
\end{equation}
In particular,
$\|\nabla F_{\sigma_k}(x_k)-\nabla F_\eps(x_k)\|\to0$.

Consider the martingale $M_n:=\sum_{j=0}^{n-1}\gamma_j\xi_{j+1}$.
Its increments are orthogonal in $L^2$, and the expectation bound above
gives
\[
 \sup_n\E\|M_n\|^2
 \leq C_\xi\sup_j\E V_j(x_j)\sum_j\gamma_j^2<\infty.
\]
Thus $M_n$ converges almost surely. On a sample path where this holds,
set
\[
 A_n:=\sum_{j=n}^\infty\gamma_j\|\nabla F_{\sigma_j}(x_j)\|^2,
 \qquad R_n:=\sup_{m\geq n}\|M_m-M_n\|.
\]
Both tend to zero. Fix $T>0$, and let $m(n)$ be the first index such
that $\sum_{j=n}^{m(n)-1}\gamma_j\geq T$. For
$n\leq l\leq m(n)$, summing \eqref{eq:ered-moving-recursion} and
applying Cauchy--Schwarz gives
\begin{align*}
 \|x_l-x_n\|
 &\leq\sum_{j=n}^{l-1}\gamma_j\|\nabla F_{\sigma_j}(x_j)\|+R_n\\
 &\leq\left(\sum_{j=n}^{l-1}\gamma_j\right)^{1/2}
          \left(\sum_{j=n}^{l-1}\gamma_j\|\nabla F_{\sigma_j}(x_j)\|^2\right)^{1/2}+R_n.
\end{align*}
The first sum is at most $T+\bar\gamma$ by the minimality of $m(n)$,
and the second is at most $A_n$. Therefore
\begin{equation}
 \max_{n\leq l\leq m(n)}\|x_l-x_n\|
 \leq\sqrt{(T+\bar\gamma)A_n}+R_n\longrightarrow0.
 \label{eq:ered-window}
\end{equation}
The global $L$-Lipschitz bound and
\eqref{eq:ered-noise-regularity} then imply
\[
 \max_{n\leq l\leq m(n)}\|h_l-h_n\|
 \leq L\max_{n\leq l\leq m(n)}\|x_l-x_n\|
    +2\lambda L^\sigma_{g,\cK}
       \sup_{j\geq n}|\sigma_j-\eps|\longrightarrow0.
\]
If $\|h_n\|\not\to0$, there are $\eta>0$ and arbitrarily large
$n$ with $\|h_n\|\geq\eta$. For every sufficiently large such $n$,
the last estimate gives $\|h_l\|\geq\eta/2$ throughout the window,
so
\[
 A_n\geq\sum_{j=n}^{m(n)-1}\gamma_j\|\nabla F_{\sigma_j}(x_j)\|^2
     \geq\frac{\eta^2}{4}
             \sum_{j=n}^{m(n)-1}\gamma_j
     \geq\frac{T\eta^2}{4},
\]
contradicting $A_n\to0$. Hence $\|\nabla F_{\sigma_k}(x_k)\|\to0$, and
\eqref{eq:ered-noise-regularity} gives
$\|\nabla F_\eps(x_k)\|\to0$.

Finally, \eqref{eq:ered-potential-variation} and boundedness imply
$| F_{\sigma_k}(x_k)- F_\eps(x_k)|\to0$. Since $V_k(x_k)$
converges, so does $ F_\eps(x_k)$. Also $\gamma_k\to0$,
$(\nabla F_{\sigma_k}(x_k))$ is bounded, and convergence of $M_n$ implies
$\gamma_k\xi_{k+1}\to0$. The recursion therefore gives
$\|x_{k+1}-x_k\|\to0$. Continuity of $\nabla F_\eps$ makes every
cluster point critical, and compactness of $\cK$ gives
$\operatorname{dist}(x_k,\operatorname{crit} F_\eps)\to0$.
\end{proof}

\subsection{Proof of Theorem~\ref{thm:prox-pgd}}
\label{app:proof-prox-pgd}

\begin{proof}
We first spell out the consequences of Assumption~\ref{ass:regular-proximal-denoiser}.
In both cases, $D_\sigma=\operatorname{prox}_{\phi_\sigma}$ with the
potential in \eqref{eq:phi_sigma}. For a gradient-step denoiser,
Proposition~\ref{prop:gs-proximal-regularity} shows that $D_\sigma$
is a global diffeomorphism and $\phi_\sigma$ is jointly $\mathcal C^2$
on $U\times I_\eps$, where $U=\R^d$. For an exact MMSE
denoiser, Proposition~\ref{prop:various_mmse_results} gives the same
joint regularity on the common open convex set
$U=\operatorname{int}(\operatorname{conv}(\operatorname{supp}P_X))$,
and $D_\sigma$ maps $\R^d$ into $U$. In this case $U$ is bounded
because the prior has compact support. In either case, $\phi_\eps$
is subanalytic on $U$ by Appendix~\ref{app:denoiser-verification}.
Set $\Psi_\sigma:=\gamma F_\sigma=\gamma f+\phi_\sigma$;
in particular, $\Psi_\eps=\gamma F_\eps$.

We begin with descent for the objective used at iteration $k$.
Write $z_k=x_k-\gamma\nabla f(x_k)$. Since
$x_{k+1}=\operatorname{prox}_{\phi_{\sigma_k}}(z_k)$, comparison
with $x_k\in U$ gives
\begin{equation}
 \phi_{\sigma_k}(x_{k+1})
 \leq\phi_{\sigma_k}(x_k)
 -\gamma\langle\nabla f(x_k),x_{k+1}-x_k\rangle
 -\frac12\|x_{k+1}-x_k\|^2.
 \label{eq:prox-pgd-proximal-descent}
\end{equation}
The descent lemma for $f$ gives
\begin{equation}
 \gamma f(x_{k+1})
 \leq\gamma f(x_k)
 +\gamma\langle\nabla f(x_k),x_{k+1}-x_k\rangle
 +\frac{\gamma L_f}{2}\|x_{k+1}-x_k\|^2.
 \label{eq:prox-pgd-forward-descent}
\end{equation}
Adding the two inequalities, we obtain
\begin{equation}
 \Psi_{\sigma_k}(x_{k+1})
 \leq\Psi_{\sigma_k}(x_k)-m\|x_{k+1}-x_k\|^2,
 \qquad m:=\frac{1-\gamma L_f}{2}>0.
 \label{eq:prox-pgd-current-descent}
\end{equation}

We next prove boundedness, before using constants on a compact set.
In the MMSE case, $x_{k+1}=D_{\sigma_k}(z_k)\in U$ for every $k$.
We may start the descent argument at $k=1$, since discarding the first
update does not affect the asymptotic conclusions or the rate from
indices $1$ to $K$. For notational simplicity, we relabel this tail
from $k=0$. Thus the iterates used below lie in the bounded
set $U$. In the gradient-step case,
Proposition~\ref{prop:gs-proximal-regularity} gives
\begin{equation}
 \phi_\sigma(x)\geq g_\sigma(x)
 \geq c_0\|x\|^2-C_0.
 \label{eq:prox-pgd-coercivity}
\end{equation}
The same proposition gives the global variation estimate, uniformly
for $x\in\R^d$ and $\sigma,\sigma'\in I_\eps$,
\begin{equation}
 |\phi_{\sigma'}(x)-\phi_\sigma(x)|
 \leq C|\sigma'-\sigma|(1+\|x\|^2).
 \label{eq:prox-pgd-global-variation}
\end{equation}
Subtracting a lower bound from $f$, we may assume $f\geq0$.
Set $V_k(x):=1+C_0+\Psi_{\sigma_k}(x)$. From
\eqref{eq:prox-pgd-coercivity},
$V_k(x)\geq1+\gamma f(x)+c_0\|x\|^2$, and hence
$1+\|x\|^2\leq(1+c_0^{-1})V_k(x)$. Therefore,
\[
V_{k+1}(x)-V_k(x)
=\phi_{\sigma_{k+1}}(x)-\phi_{\sigma_k}(x)
\leq C(1+c_0^{-1})|\sigma_{k+1}-\sigma_k|V_k(x).
\]
Thus
\[
 V_{k+1}(x)\leq(1+a_k)V_k(x),
 \qquad a_k=C|\sigma_{k+1}-\sigma_k|,
\]
after enlarging $C$.
The sequence $(a_k)$ is summable by monotonicity of $(\sigma_k)$.
Combining this bound with
\eqref{eq:prox-pgd-current-descent} and iterating yields
$\sup_kV_k(x_k)<\infty$. The coercivity in
\eqref{eq:prox-pgd-coercivity} then gives
$\sup_k\|x_k\|<\infty$.

We can now choose a compact convex set $\cK\subset U$ containing all
iterates. In the gradient-step case, take a sufficiently large closed
ball. For MMSE, the inputs $z_k$ are bounded because the iterates,
stepsizes, and $\nabla f(x_k)$ are bounded. Let $B_R$ be a closed ball
containing them. Joint continuity of $D_\sigma$ makes
$S:=\{D_\sigma(z):z\in B_R,\ \sigma\in I_\eps\}$ a compact subset
of $U$. Since $U$ is convex, the convex hull of $S\cup\{x_0\}$ is a
compact subset of $U$ and contains every iterate; take it as $\cK$.

Joint regularity of $\phi_\sigma$ on $\cK\times I_\eps$ gives finite
constants
\[
 L_{\phi,\cK}:=\sup_{\substack{x\in\cK\\\sigma\in I_\eps}}
       |\partial_\sigma\phi_\sigma(x)|,
 \qquad
 L_{\nabla\phi,\cK}:=
 \sup_{\substack{x\in\cK\\\sigma\in I_\eps}}
       \|\partial_\sigma\nabla\phi_\sigma(x)\|.
\]
Continuity on the compact parameter set also gives the uniform lower bound
\[
\Psi_\star:=
 \inf_{x\in\cK,\,\sigma\in I_\eps}\Psi_\sigma(x)>-\infty.
\]
For every $x\in\cK$, the mean-value theorem thus gives
\begin{equation}
 |\Psi_{\sigma_{k+1}}(x)-\Psi_{\sigma_k}(x)|
 \leq L_{\phi,\cK}|\sigma_{k+1}-\sigma_k|=:\delta_k.
 \label{eq:prox-pgd-schedule-error}
\end{equation}
Since $\sum_k\delta_k<\infty$ by monotonicity of $(\sigma_k)$,
combining this bound at $x_{k+1}$
with \eqref{eq:prox-pgd-current-descent} yields
\begin{equation}
 \Psi_{\sigma_{k+1}}(x_{k+1})
 \leq\Psi_{\sigma_k}(x_k)-m\|x_{k+1}-x_k\|^2+\delta_k.
 \label{eq:prox-pgd-quasi-descent}
\end{equation}
Define $r_k:=\sum_{j=k}^\infty\delta_j$. Then
$\Psi_{\sigma_k}(x_k)+r_k$ is nonincreasing and bounded below, so it
converges. As $r_k\to0$, $\Psi_{\sigma_k}(x_k)$ converges, and summing
\eqref{eq:prox-pgd-quasi-descent} gives
\begin{equation}
 m\sum_{k=0}^K\|x_{k+1}-x_k\|^2
 \leq\Psi_{\sigma_0}(x_0)-\Psi_\star+\sum_{k=0}^K\delta_k.
 \label{eq:prox-pgd-energy-bound}
\end{equation}
Hence $\sum_k\|x_{k+1}-x_k\|^2<\infty$ and
$\|x_{k+1}-x_k\|\to0$. Moreover,
\[
 |\Psi_{\sigma_k}(x_k)-\Psi_\eps(x_k)|
 \leq L_{\phi,\cK}|\sigma_k-\eps|\longrightarrow0.
\]
Thus $F_\eps(x_k)=\Psi_\eps(x_k)/\gamma$ converges.

We next prove asymptotic stationarity, as in the RED--GD proof (Theorem~\ref{thm:gd-pnp}).
First-order optimality of the proximal step gives
\[
 x_{k+1}-x_k+\gamma\nabla f(x_k)
       +\nabla\phi_{\sigma_k}(x_{k+1})=0.
\]
Consequently,
\begin{equation}
 \nabla\Psi_{\sigma_k}(x_{k+1})
 =\gamma\bigl(\nabla f(x_{k+1})-\nabla f(x_k)\bigr)
       -(x_{k+1}-x_k),
 \quad
 \|\nabla\Psi_{\sigma_k}(x_{k+1})\|
 \leq(1+\gamma L_f)\|x_{k+1}-x_k\|.
 \label{eq:prox-pgd-optimality-gradient}
\end{equation}
Moreover, the change of objective from iteration $k-1$ to $k$
satisfies
\[
 \|\nabla\Psi_{\sigma_k}(x_k)
       -\nabla\Psi_{\sigma_{k-1}}(x_k)\|
 \leq L_{\nabla\phi,\cK}|\sigma_k-\sigma_{k-1}|=:b_k.
\]
Since $(\sigma_k)$ has finite total variation, $\sum_k b_k^2<\infty$.
Therefore, for $k\geq1$,
\begin{equation}
 \|\nabla F_{\sigma_k}(x_k)\|
 =\frac{1}{\gamma}\|\nabla\Psi_{\sigma_k}(x_k)\|
 \leq\frac{(1+\gamma L_f)\|x_k-x_{k-1}\|+b_k}{\gamma}.
 \label{eq:prox-pgd-gradient-bound}
\end{equation}
 Together with
\eqref{eq:prox-pgd-energy-bound}, this proves
$\sum_{k\geq1}\|\nabla F_{\sigma_k}(x_k)\|^2<\infty$.
In particular $\|\nabla F_{\sigma_k}(x_k)\|\to0$, and
\begin{equation}
 \min_{1\leq k\leq K}\|\nabla F_{\sigma_k}(x_k)\|^2
 \leq\frac{1}{K}\sum_{k=1}^\infty
              \|\nabla F_{\sigma_k}(x_k)\|^2
 =\mathcal O(K^{-1}).
 \label{eq:prox-pgd-current-rate-proof}
\end{equation}
Finally, on $\cK$, the terminal and moving gradients satisfy
\begin{equation}
 \|\nabla F_{\sigma_k}(x)-\nabla F_\eps(x)\|
 \leq\frac{L_{\nabla\phi,\cK}}{\gamma}
       |\sigma_k-\eps|\longrightarrow0
 \quad\text{uniformly in }x\in\cK.
 \label{eq:prox-pgd-terminal-gradient-transfer}
\end{equation}
Thus $\|\nabla F_\eps(x_k)\|\to0$. If
$x_{k_j}\to\bar x$, continuity of $\nabla F_\eps$ gives
$\nabla F_\eps(\bar x)=0$, so every cluster point is critical.

\paragraph{Point convergence.}
Assume now the additional decay condition
$|\sigma_k-\eps|=\mathcal O(k^{-\alpha})$ with $\alpha>1$. Set
\[
 e_k:=L_{\nabla\phi,\cK}|\sigma_k-\eps|.
\]
Then $\sup_{x\in\cK}\|\nabla(\Psi_\eps-\Psi_{\sigma_k})(x)\|
\leq e_k=\mathcal O(k^{-\alpha})$, and
$\sum_ke_k<\infty$, $\sum_ke_k^2<\infty$.
Since $\cK$ is convex, integration along $[x_k,x_{k+1}]$ and
\eqref{eq:prox-pgd-current-descent} give
\begin{equation}
 \Psi_\eps(x_{k+1})
 \leq\Psi_\eps(x_k)-m\|x_{k+1}-x_k\|^2
       +e_k\|x_{k+1}-x_k\|
 \leq\Psi_\eps(x_k)-\frac m2\|x_{k+1}-x_k\|^2
       +\frac{e_k^2}{2m}.
 \label{eq:prox-pgd-kl-descent}
\end{equation}
As in the RED--GD proof, choose an even integer
$\theta>(2\alpha-1)/(2\alpha-2)$ and define
\[
 t_k:=\left(\frac{\theta}{2m}
             \sum_{j=k}^\infty e_j^2\right)^{1/\theta},
 \qquad
 \Phi(x,t):=\Psi_\eps(x)+\frac{t^\theta}{\theta}.
\]
Then \eqref{eq:prox-pgd-kl-descent} becomes
\[
 \Phi(x_{k+1},t_{k+1})+\frac m2\|x_{k+1}-x_k\|^2
 \leq\Phi(x_k,t_k),
\]
which verifies the sufficient-decrease condition~(H1) of
\citet[Assumption~H]{ochs2019inexact}. By
\eqref{eq:prox-pgd-optimality-gradient},
\[
 \|\nabla\Phi(x_{k+1},t_{k+1})\|
 \leq(1+\gamma L_f)\|x_{k+1}-x_k\|
       +e_k+t_{k+1}^{\theta-1}.
\]
The choice of $\theta$ ensures
$t_k^{\theta-1}
 =\mathcal O(k^{-(2\alpha-1)(\theta-1)/\theta})$ is
summable. Thus the relative-error condition~(H2) holds.
The bounded sequence $(x_k,t_k)$ has a convergent subsequence;
continuity of $\Phi$ on $U\times\R$ makes it
$\Phi$-attentively convergent, as required in~(H3).
The remaining distance and parameter conditions follow by taking
$\|x_{k+1}-x_k\|$ and constant positive coefficients.

Finally, $\Psi_\eps$ is subanalytic on $U$, hence $\Phi$ is KL
there. To use the global formulation of the KL theorem, choose a
compact polytope $P$ such that
$\cK\subset\operatorname{int}P\subset P\subset U$ and extend
$\Phi|_{P\times\R}$ by $+\infty$ outside $P\times\R$.
This extension is proper, lower semicontinuous, subanalytic and
bounded below; it agrees with $\Phi$ near all iterates and cluster
points. Therefore \citet[Theorem~10]{ochs2019inexact} gives
\[
 \sum_k\|x_{k+1}-x_k\|<\infty,
 \qquad x_k\to x_\star,
 \qquad \nabla\Psi_\eps(x_\star)
       =\gamma\nabla F_\eps(x_\star)=0.
\]
This proves finite length and convergence to a critical point of $F_\eps$.
\end{proof}

\subsection{SNOPnP and PnP--Flow}
\label{app:flow-change-variables}

PnP--Flow~\citep{martin2025pnp} interpolates a gradient step with fresh
Gaussian noise before applying the flow denoiser:
\begin{equation}
 x_{k+1}=\widetilde D_{t_k}\!\left(
 t_k\bigl(x_k-\gamma\nabla f(x_k)\bigr)
 +(1-t_k)Z_{k+1}\right),
 \qquad t_k\in(0,1),\quad Z_{k+1}\sim\N(0,I_d).
 \label{eq:PnP--Flow-update}
\end{equation}
Here $\widetilde D_t$ is a denoiser that takes as input the flow-matching time interpolation between image and noise. For $\sigma>0$, set $s(\sigma)=1/(1+\sigma)$ and
$D_\sigma(u)=\widetilde D_{s(\sigma)}(s(\sigma)u)$.
With this change of variable, $D_\sigma$ is a proper image-space Gaussian denoiser that takes the rescaled image-space input $u$.
Writing $\sigma_k=(1-t_k)/t_k$ gives $t_k=s(\sigma_k)$ and
$t_kz+(1-t_k)Z=t_k(z+\sigma_kZ)$. Thus \eqref{eq:PnP--Flow-update}
becomes
\begin{equation*}
 x_{k+1}=D_{\sigma_k}\bigl(
 x_k-\gamma\nabla f(x_k)+\sigma_kZ_{k+1}\bigr).
\end{equation*}
which is exactly the SNOPnP algorithm proposed in \cite{renaud2026equivariant}. 

Assuming
that the rescaled denoiser is proximal, the PnP--Flow / SNOPnP update reads
\[
x_{k+1}
=
\operatorname{prox}_{\phi_{\sigma_k}}
\bigl(x_k-\gamma\nabla f(x_k)+\sigma_k Z_{k+1}\bigr),
\qquad Z_{k+1}\sim\mathcal N(0,I_d).
\]

\paragraph{Decoupling the two noise levels.}
In the above update, the injected
noise has amplitude $\sigma_k$, the same level used by $D_{\sigma_k}$.
In the proximal comparison below, this perturbation produces a term
proportional to $\sigma_k^2\|Z_{k+1}\|^2$. Our descent
argument in the proof applies when the sum of these stochastic terms is finite. If
$\sigma_k\to\eps>0$, this condition fails. 
We thus decouple the stochastic part and the noise level of the denoiser: we use a different schedule $\tau_k$ for each:
\begin{equation}
\label{eq:decoupled}
x_{k+1}
=
\operatorname{prox}_{\phi_{\sigma_k}}
\bigl(x_k-\gamma\nabla f(x_k)+\tau_k Z_{k+1}\bigr),
\qquad Z_{k+1}\sim\mathcal N(0,I_d).
\end{equation}
The condition
$\sum_k\tau_k^2<\infty$ then controls the perturbations while
$\sigma_k\to\eps>0$ retains the terminal objective
$F_\eps=f+\gamma^{-1}\phi_\eps$.

\begin{theorem}[Convergence of decoupled SNOPnP / PnP--Flow]
\label{thm:PnP--Flow}
Suppose Assumptions~\ref{ass:noise-schedule},
\ref{ass:prox-lipschitz-x},~\ref{ass:denoiser_gs}, and
\ref{ass:regular-proximal-denoiser} hold. Use a constant stepsize
$\gamma>0$ with $\gamma L_f<1$, and suppose
$\sum_k\tau_k^2<\infty$. Let
$F_\sigma=f+\gamma^{-1}\phi_\sigma$.
Almost surely, 
$\|\nabla F_\eps(x_k)\|\to0$,
$\|x_{k+1}-x_k\|\to0$, and every cluster point is critical for
$F_\eps$. In particular,
$\dist(x_k,\crit F_\eps)\to0$.
Moreover, almost surely,
\begin{equation}
 \min_{1\leq k\leq K}\|\nabla F_{\sigma_k}(x_k)\|^2
 =\mathcal O\!\left(K^{-1}\right),\qquad K\geq1.
 \label{eq:PnP--Flow-current-rate}
\end{equation}
\end{theorem}

\begin{proof}
The consequences of Assumption~\ref{ass:regular-proximal-denoiser}
follow from Propositions~\ref{prop:gs-proximal-regularity}
and~\ref{prop:various_mmse_results}, as in the proof of
Theorem~\ref{thm:prox-pgd}. In particular,
$D_\sigma=\operatorname{prox}_{\phi_\sigma}$, and
$\phi_\sigma$ is jointly $\mathcal C^2$ on $U\times I_\eps$,
where $U=\R^d$ for a contractive gradient-step denoiser and
$U=\operatorname{int}(\operatorname{conv}(\operatorname{supp}P_X))$
for the exact MMSE denoiser. Set
\[
 \Psi_\sigma:=\gamma F_\sigma=\gamma f+\phi_\sigma,
 \qquad
 e_{k+1}:=\tau_kZ_{k+1}.
\]
Since
\[
 \E\sum_k\|e_{k+1}\|^2=d\sum_k\tau_k^2<\infty,
\]
we have $\sum_k\|e_{k+1}\|^2<\infty$ almost
surely. We work on a realization with this property; all subsequent
bounds are pathwise.

We first establish descent for the moving objective. For a
gradient-step denoiser the following comparison holds from $k=0$;
for MMSE it holds from $k=1$, when $x_k\in U$. By definition of the proximal operator at input
$z_k=x_k-\gamma\nabla f(x_k)+e_{k+1}$, we have
\[
 \phi_{\sigma_k}(x_{k+1})
 \leq\phi_{\sigma_k}(x_k)
 -\gamma\langle\nabla f(x_k),x_{k+1}-x_k\rangle
 -\frac12\|x_{k+1}-x_k\|^2
 +\langle e_{k+1},x_{k+1}-x_k\rangle.
\]
Adding the descent lemma for $f$ yields
\begin{equation}
 \Psi_{\sigma_k}(x_{k+1})
 \leq\Psi_{\sigma_k}(x_k)-m\|x_{k+1}-x_k\|^2
       +\langle e_{k+1},x_{k+1}-x_k\rangle,
 \qquad m:=\frac{1-\gamma L_f}{2}>0.
 \label{eq:flow-noisy-descent}
\end{equation}
Young's inequality then gives
\begin{equation}
 \Psi_{\sigma_k}(x_{k+1})
 \leq\Psi_{\sigma_k}(x_k)-\frac m2\|x_{k+1}-x_k\|^2
       +\frac{1}{2m}\|e_{k+1}\|^2.
 \label{eq:flow-current-descent}
\end{equation}

We now prove boundedness of the iterates. In the MMSE case, every iterate from $x_1$ onward belongs
to the bounded set $D_{\sigma_k}(\R^d)=U$ by
Proposition~\ref{prop:various_mmse_results}. Up to changing the start at $k=1$, we can assume bounded iterates. In the gradient-step
case, the estimates established under the same assumptions in the
proof of Theorem~\ref{thm:prox-pgd} give
\[
 \phi_\sigma(x)\geq c_0\|x\|^2-C_0,
 \qquad
 |\phi_{\sigma'}(x)-\phi_\sigma(x)|
 \leq C|\sigma'-\sigma|(1+\|x\|^2).
\]
After subtracting a lower bound from $f$, assume $f\geq0$.
As in the PnP--PGD proof, set
$V_k(x):=1+C_0+\Psi_{\sigma_k}(x)$, so that
$V_k(x)\geq1+\gamma f(x)+c_0\|x\|^2$ and
\[
 V_{k+1}(x)\leq(1+a_k)V_k(x),
 \qquad
 a_k=C|\sigma_{k+1}-\sigma_k|,
 \qquad \sum_k a_k<\infty.
\]
Combining this bound with \eqref{eq:flow-current-descent} gives
\[
 V_{k+1}(x_{k+1})
 \leq(1+a_k)\left(V_k(x_k)
                    +\frac{\|e_{k+1}\|^2}{2m}\right).
\]
Since $\prod_k(1+a_k)<\infty$ and $\sum_k\|e_{k+1}\|^2<\infty$,
iteration shows that $\sup_kV_k(x_k)<\infty$. Uniform coercivity
then gives $\sup_k\|x_k\|<\infty$.

As in the PnP--PGD proof, choose a compact convex set
$\cK\subset U$ containing $x_k$ for every $k\geq1$. For gradient-step denoisers,
a sufficiently large closed ball suffices. For MMSE, the proximal
inputs $z_k$, the iterates and $\nabla f(x_k)$ are
bounded. Joint continuity of
$D_\sigma$ makes the image of a closed ball containing all $z_k$,
over $\sigma\in I_\eps$, a compact subset $S$ of the open convex set $U$. Thus
the convex hull of $S$ is a compact subset of
$U$ and contains all iterates; we take it as $\cK$.

Joint regularity on $\cK\times I_\eps$ gives finite constants
\[
 L_{\phi,\cK}:=\sup_{\substack{x\in\cK\\\sigma\in I_\eps}}
       |\partial_\sigma\phi_\sigma(x)|,
 \quad
 L_{\nabla\phi,\cK}:=
       \sup_{\substack{x\in\cK\\\sigma\in I_\eps}}
       \|\partial_\sigma\nabla\phi_\sigma(x)\|.
\]
Set $\Psi_\star:=\inf_{x\in\cK,\,\sigma\in I_\eps}
\Psi_\sigma(x)>-\infty$. For $x\in\cK$,
\[
 |\Psi_{\sigma_{k+1}}(x)-\Psi_{\sigma_k}(x)|
 \leq L_{\phi,\cK}|\sigma_{k+1}-\sigma_k|.
\]
For $k\geq1$, define the summable error
\[
 \delta_k:=L_{\phi,\cK}|\sigma_{k+1}-\sigma_k|
 +\frac{\|e_{k+1}\|^2}{2m}.
\]
Combining the above bound at $x_{k+1}$ with
\eqref{eq:flow-current-descent} gives
\begin{equation}
 \Psi_{\sigma_{k+1}}(x_{k+1})
 \leq\Psi_{\sigma_k}(x_k)-\frac m2\|x_{k+1}-x_k\|^2+\delta_k.
 \label{eq:flow-pathwise-descent}
\end{equation}
With
$r_k:=\sum_{j=k}^\infty\delta_j$, the sequence
$\Psi_{\sigma_k}(x_k)+r_k$ is nonincreasing for $k\geq1$ and
bounded below by $\Psi_\star$. Hence $\Psi_{\sigma_k}(x_k)$ converges and
\begin{equation}
 \sum_{k=1}^\infty\|x_{k+1}-x_k\|^2<\infty.
 \label{eq:flow-step-energy}
\end{equation}
In particular, $\|x_{k+1}-x_k\|\to0$. Moreover,
\[
 |\Psi_{\sigma_k}(x_k)-\Psi_\eps(x_k)|
 \leq L_{\phi,\cK}|\sigma_k-\eps|\longrightarrow0,
\]
so $F_\eps(x_k)=\Psi_\eps(x_k)/\gamma$ converges.

It remains to prove stationarity. First-order optimality of the noisy
proximal step gives
\[
 x_{k+1}-x_k+\gamma\nabla f(x_k)
       +\nabla\phi_{\sigma_k}(x_{k+1})-e_{k+1}=0.
\]
Consequently,
\begin{equation}
 \|\nabla\Psi_{\sigma_k}(x_{k+1})\|
 \leq(1+\gamma L_f)\|x_{k+1}-x_k\|
       +\|e_{k+1}\|.
 \label{eq:flow-optimality-gradient}
\end{equation}
At $x_k$ for $k\geq2$, the change of moving objective satisfies
\[
 \|\nabla\Psi_{\sigma_k}(x_k)
       -\nabla\Psi_{\sigma_{k-1}}(x_k)\|
 \leq L_{\nabla\phi,\cK}|\sigma_k-\sigma_{k-1}|=:b_k.
\]
Thus, for $k\geq2$,
\[
 \|\nabla F_{\sigma_k}(x_k)\|
 \leq\frac{(1+\gamma L_f)\|x_k-x_{k-1}\|
              +\|e_k\|+b_k}{\gamma}.
\]
The finite total variation of $(\sigma_k)$ makes
$\sum_k b_k^2<\infty$. Together with
\eqref{eq:flow-step-energy} and $\sum_k\|e_{k+1}\|^2<\infty$,
this proves $\sum_{k\geq1}\|\nabla F_{\sigma_k}(x_k)\|^2<\infty$.
It follows that $\|\nabla F_{\sigma_k}(x_k)\|\to0$ and
\eqref{eq:PnP--Flow-current-rate}.
Moreover, 
\[
 \sup_{x\in\cK}\|\nabla F_{\sigma_k}(x)-\nabla F_\eps(x)\|
 \leq\frac{L_{\nabla\phi,\cK}}{\gamma}
       |\sigma_k-\eps|\longrightarrow0.
\]
Hence $\|\nabla F_\eps(x_k)\|\to0$. Continuity of
$\nabla F_\eps$ makes every cluster point critical, and
compactness of $\cK$ gives $\dist(x_k,\crit F_\eps)\to0$.
\end{proof}

\section{Experimental details}
\label{app:experiments}

\subsection{Denoisers: architectures, training and results}
\label{app:denoiser_training}
\label{app:denoisers}

The four denoisers of Section~\ref{sec:experiments} are gradient-step denoisers $D_\sigma=\id-\nabla g_\sigma$ with $$g_\sigma(x)=\frac12\|x-N_\sigma(x)\|^2$$ There are two training phases that we detail below: a \emph{gradient-step phase}, which trains $D_\sigma$ as a gradient-step denoiser and a \emph{proximal phase} started from its output, which adds the penalty on $\|\nabla^2 g_\sigma\|_2$ to get a proximal denoiser. %

\paragraph{Architectures and initialization.}
On natural images, $N_\sigma$ is the DRUNet of~\citet{zhang2021dpir} in the configuration of GS--DRUNet~\citep{hurault2022gradient}: a U-Net with four scales of $64$, $128$, $256$ and $512$ channels, two residual blocks per scale, the noise level given as a constant input channel, $17.0$M parameters, and softplus activations instead of ELU, so that $N_\sigma$ is $\mathcal C^\infty$ and the Hessian $\nabla^2 g_\sigma$ on which the proximal analysis rests is continuous (ELU is only $\mathcal C^1$). It is initialized with the weights of the published GS--DRUNet, trained for $\sigma\in[0,50/255]$. 

On FFHQ face images, $N_\sigma$ is the U-Net of~\citet{dhariwal2021diffusion} trained by~\citet{choi2021ilvr} as an $\epsilon$-prediction diffusion model on FFHQ at $256\times256$ and used by DiffPIR~\citep{zhu2023denoising}: $128$ base channels, one residual block per level, self-attention at one resolution, $93.6$M parameters. 

We normalize clean images in $[0,1]$ intensity units. For a
noisy input $y=x+\sigma z$, with $z\sim\mathcal N(0,I)$, conversion to
the diffusion model's $[-1,1]$ coordinates gives
$2y-1=(2x-1)+2\sigma z$. To put this input in the model's
variance-preserving form, we set
$\bar\alpha=(1+4\sigma^2)^{-1}$ and the network
receives $\sqrt{\bar\alpha}(2y-1)$ at time $t(\sigma)$, and we use its
noise prediction to define
\begin{equation*}
N_\sigma(y)=y-\sigma\,\epsilon_\theta\bigl(
    \sqrt{\bar\alpha}(2y-1),\,t(\sigma)\bigr).
\end{equation*}
We obtain $t(\sigma)$ from the continuous variance-preserving schedule,
rather than using the discrete timestep lookup of the original DDPM
implementation, to preserve regularity in $\sigma$.

\paragraph{Training data.}
Our DRUNet models are trained on the original training set of the published DRUNet from~\cite{zhang2021dpir} (BSD400, the Waterloo Exploration Database, DIV2K and Flickr2K, $8\,666$ images), on random $128\times128$ crops taken at the native scale of the images with random horizontal and vertical flips; the last $16$ images are held out for the validation of the training. DiffUNet models are trained on FFHQ at $128\times128$ with random horizontal flips.

\paragraph{Gradient-step training phase.}
The network minimizes the weighted $L^2$ denoising loss
\begin{equation}
\label{eq:gs-loss}
\mathcal L_{\rm GS}(\theta)=\E_{x,\sigma,z}\Big[w(\sigma)\,\big\|D_\sigma(x+\sigma z)-x\big\|^2\Big],\qquad D_\sigma=\id-\nabla g_\sigma,\quad w(\sigma)=\frac{\sigma^2+\sigma_{\rm d}^2}{\sigma^2\sigma_{\rm d}^2},
\end{equation}
where $\|\cdot\|^2$ is the mean squared error over the pixels, $\theta$ are the parameters of $N_\sigma$ and $w$ is the weighting proposed by~\citet{karras2022elucidating} with $\sigma_{\rm d}=0.25$ (the standard deviation of the data for images in $[0,1]$), which equalizes the contribution of the noise levels: the loss is of order $1$ at every $\sigma$.   On natural images $\sigma$ is log-uniform on $[0.002,2]$, so that every decade of noise receives the same number of samples; on faces it follows the log-normal distribution of EDM, $\ln(2\sigma)\sim\mathcal N(-1.2,1.2^2)$, clipped to $[0.002,5]$.

The gradient $\nabla g_\sigma(y)$ is computed exactly by automatic differentiation, so a training step differentiates twice through the network (double backpropagation), as in~\citet{hurault2022gradient}. Both phases use AdamW at learning rate $10^{-4}$ after a linear warm-up of $1000$ steps. GS--DRUNet is trained for $10$k steps at global batch $64$.  GS--DiffUNet is trained for $4$k steps at global batch $64$ (about four epochs of FFHQ). 

Table~\ref{tab:spectrum} reports the PSNR performance of the original and finetuned models at different noise levels. It shows that the conservative gradient-step parametrization does not reduce denoising performance: initialized from the pretrained DiffUNet, GS--DiffUNet matches or exceeds its PSNR on FFHQ after only $4$k training steps. Likewise, starting from the published GS--DRUNet trained only for $\sigma\leq50/255$, fine-tuning successively extends its performance to larger noise levels.

\paragraph{Proximal phase.}
Following~\citet{hurault2022proximal}, the training of the proximal denoisers is done by starting from the gradient-step checkpoints, with the additional term in the loss:
\begin{equation}
\label{eq:prox-loss}
\mathcal L_{\rm Prox}(\theta)=\mathcal L_{\rm GS}(\theta)+\mu\,\E\Big[\sum_{i=1}^{n_p}\omega_i\,\max\big(0, \|\nabla^2 g_{\sigma_i}(y_i)\|_2-(1-\delta)\big)^2\Big],
\end{equation}
where $0<\delta<1$ is the margin, $\mu$ the weight of the penalty and the weights $\omega_i$ chosen as a softmax over the spectral norms of the current batch: we find it advantageous to penalize stronger the higher spectral norms instead of the mean like~\citet{hurault2022proximal}. Due to time and memory constraint, the penalty is evaluated on the first $8$ images of each batch.

The spectral norm is estimated by the Lanczos method~\citep{golub2013matrix} with $m=20$ iterations. The Lanczos iteration is warm-started from the Ritz vector of the previous step (with $5\%$ of Gaussian jitter), and stopped early when the estimate moves by less than $10^{-3}$. We find Lanczos iterations to converge faster than the power iteration of~\citet{hurault2022proximal}.  We set $\mu=1$, $\delta=0.2$, Prox--DRUNet runs $4$k steps and Prox--DiffUNet runs $10$k steps. 

Table~\ref{tab:spectrum} reports the maximum spectral norm of the Hessian of $g_\sigma$ and the PSNR denoising performance of the original and finetuned models, at different noise levels. On both datasets, proximal training sharply reduces the maximum Hessian spectral norm, bringing it below $1$ for small noise levels, while keeping it close to $1$ at larger noise levels. Handling the constraint on small noise levels is more important because the annealed schedules approach them in the late iterations, where the denoiser governs the limiting behavior of the algorithm.

\begin{table}[h]
\centering\small
\setlength{\tabcolsep}{3pt}
\caption{Denoising PSNR and spectral norm $\| \nabla^2 g_\sigma \|$ on $16$ noisy validation images. GS and Prox denote the gradient-step and proximal denoisers. Bold norms are below $1$.}
\label{tab:spectrum}
\begin{minipage}[t]{\linewidth}
\centering
\textbf{FFHQ (DiffUNet)}\par\smallskip
\begin{tabular}{@{}lcccccc@{}}
\toprule
$\sigma$ & 0.01 & 0.05 & 0.2 & 0.5 & 1 & 5\\
\midrule
\multicolumn{7}{@{}l}{\emph{PSNR (dB)}}\\
Pretrained DiffUNet~\citep{choi2021ilvr} & 44.61 & 35.54 & 28.60 & 24.45 & 21.49 & 15.48\\
GS--DiffUNet (fine-tuned to $\sigma\leq5$) & 44.54 & 35.69 & 28.87 & 24.87 & 22.08 & 15.75\\
Prox--DiffUNet & 44.42 & 35.33 & 28.28 & 24.36 & 21.72 & 16.12\\
\midrule
\multicolumn{7}{@{}l}{\emph{$\max_x\|\nabla^2 g_\sigma(x)\|_2$}}\\
GS--DiffUNet & 1.090 & 2.117 & 4.155 & 3.062 & 3.096 & 1.207\\
Prox--DiffUNet & \textbf{0.902} & \textbf{0.969} & \textbf{0.982} & 1.008 & 1.025 & 1.035\\
\bottomrule
\end{tabular}
\end{minipage}\par\medskip
\begin{minipage}[t]{\linewidth}
\centering
\textbf{CBSD68 (DRUNet)}\par\smallskip
\begin{tabular}{@{}lcccccc@{}}
\toprule
$\sigma$ & 0.01 & 0.05 & 0.2 & 0.5 & 1 & 2\\
\midrule
\multicolumn{7}{@{}l}{\emph{PSNR (dB)}}\\
Published GS--DRUNet~\citep{hurault2022gradient} & 43.69 & 34.03 & 27.07 & 20.82 & 10.46 & $-0.66$\\
GS--DRUNet (fine-tuned to $\sigma\leq2$) & 43.88 & 33.97 & 27.00 & 23.24 & 20.76 & 18.61\\
Prox--DRUNet & 43.66 & 33.48 & 26.24 & 22.46 & 20.08 & 17.84\\
\midrule
\multicolumn{7}{@{}l}{\emph{$\max_x\|\nabla^2 g_\sigma(x)\|_2$}}\\
GS--DRUNet & 1.202 & 1.287 & 1.628 & 1.524 & 1.800 & 1.574\\
Prox--DRUNet & \textbf{0.882} & \textbf{0.919} & \textbf{0.951} & \textbf{0.977} & \textbf{0.989} & 1.040\\
\bottomrule
\end{tabular}
\end{minipage}

\end{table}

\paragraph{Relaxation of the proximal denoisers.}
The spectral penalty in \eqref{eq:prox-loss} encourages
$\|\nabla^2 g_\sigma\|_2<1$ on training samples but does not certify this
bound globally. The validation measurements in Table~\ref{tab:spectrum}
also slightly exceed $1$ at some noise levels. Following
\citet[Section~2.3]{hurault2023relaxed}, we therefore use at inference
the relaxed denoiser
\[
D^\rho_\sigma=(1-\rho)\Id+\rho D_\sigma
             =\Id-\rho\nabla g_\sigma
\]
If $\sup_x\|\nabla^2g_\sigma(x)\|_2<1/\rho$ uniformly over the noise
levels used, then $D^\rho_\sigma$ is a proximal map. In practice, we use $\rho = 0.7$ for all our PnP experiments. This raises
the sufficient threshold on $\|\nabla^2 g_\sigma\|_2$ from $1$ to
$1/\rho\simeq1.43$, above every Prox denoisers validation norm reported in
Table~\ref{tab:spectrum}.  This relaxation thus provides an empirical safety margin.
However, this is only a local, empirical check: we do not guarantee the
constraint for all inputs, but expect it to hold at least in the late PnP
iterations, when the denoiser input images should be close to the ones constrained by the proximal loss.

\subsection{Algorithms, schedules and energy curves}
\label{app:experiments_algos}

\paragraph{Noise schedule.}
Every algorithm runs for $K=500$ iterations with $\sigma_k=\eps+(\sigma_0-\eps)r^k$, except on sparse-view tomography where the runs need $K=1000$ to reach their floor. We set $\sigma_0=5$ on faces, $\sigma_0=1$ on natural images, and $r=0.98$ on both; these choices are shared by all algorithms.

\paragraph{Stepsizes.}
The RED algorithms use $\gamma_0=\gamma_{\rm rel}/(L_f+\lambda L_g)$ with $\gamma_{\rm rel}<2$, where $L_g$ bounds the Lipschitz constant of $\id-D_\sigma=\nabla g_\sigma$ over $\sigma\in[\eps,\sigma_0]$. Following the estimated Lipschitz constant given in Table~\ref{tab:spectrum}, we use $L_g=2.2$ for GS--DRUNet and $L_g=4.2$ for GS--DiffUNet. PnP algorithms use $\gamma=c/L_f$ with $c<1$. For SNORE and ERED, the stepsize is held constant at $\gamma_0$ until iteration step $k_w$ and then follows the decrease
\[
 \gamma_k=\gamma_0\bigl(1+(k-k_w)_+/k_0\bigr)^{-q},
 \qquad k_0=K/4,\quad 1/2<q\leq1.
\]
$k_w$ is set either at $0$ or at the first index at which $\sigma_k<2\eps$, depending on the problem. Thus $\sum_k\gamma_k=\infty$ and $\sum_k\gamma_k^2<\infty$. Delaying the decay can prevent the step from becoming small before a high initial noise level reaches its floor. This delayed decay is selected for SNORE on face demosaicing and super-resolution and for ERED on face super-resolution; everywhere else, the decay from $k=0$ is kept. Table~\ref{tab:hparams} shows the selected $q$ parameter for each concerned problem. For decoupled SNOPnP (Theorem~\ref{thm:PnP--Flow}), the injected noise follows
\[
 \tau_k=(\sigma_k-\eps)+\eps\bigl(1+(k-k_w)_+/k_0\bigr)^{-q},
 \qquad q>1/2.
\]
It initially tracks the denoiser level, then vanishes independently; $\sum_k\tau_k^2<\infty$ because $\sigma_k-\eps$ decays geometrically. The same $k_w$ convention is used in its search grid.

\paragraph{Plotted energies.} In the different figures, we plot in the middle panel the evolution of the terminal objective along iterates $F_\eps(x_k)$.
For SNORE, we estimate the potential $\tilde g_\sigma(x)=\E_Z[g_\sigma(x+\sigma Z)]$ using an average with four fresh samples at each iteration; these samples are not reused across iterations. For PnP, we need to invert the terminal denoiser at each iterate
to evaluate $\phi_\epsilon(x_k)$ using equation~\eqref{eq:pot_phi}.  We use damped fixed-point iteration warm-started at the denoiser input $u_{k-1}$ of the algorithm's own step.

\subsection{Further results on natural images}

\paragraph{Natural images.}

Table~\ref{tab:natural} compares constant and annealed noise levels on CBSD68 for inpainting, tomography, and $\times4$ super-resolution. RED methods use GS--DRUNet and PnP methods use Prox--DRUNet, both finetuned to $\sigma\leq2$. For tomography, we use a parallel-beam Radon transform with $60$ angles over $180^\circ$, $\sigma_n=0.01$, and filtered back-projection initialization. 
Annealing improves PSNR for every algorithm on all three problems. Gains are smaller for random-mask inpainting than for the large face hole: the measurements already constrain much of the coarse image structure that a large initial $\sigma$ helps recover. Tomography is more ambiguous because sparse angular sampling leaves global streaks, thus giving high gains to annealing. Figures~\ref{fig:tomo} and~\ref{fig:inpaint} show representative reconstructions and convergence curves for tomography and inpainting. Across these examples, the terminal objectives level off and the deterministic residuals decay, as predicted by the convergence analysis.

\begin{table}[h]
\centering\small
\setlength{\tabcolsep}{4.5pt}
\caption{PSNR (dB) on the $68$ CBSD68 test images ($128\times128$) for inpainting, tomography with $60$ views, and super-resolution, same layout as Table~\ref{tab:main}: constant noise level $\sigma_c$ versus annealed schedule $\sigma_k\to\eps$, both tuned, and DPIR. RED algorithms use GS--DRUNet, PnP algorithms Prox--DRUNet, both finetuned to $\sigma\le2$. Best annealed value per problem in bold.}
\label{tab:natural}

\begin{tabular}{lcccccc}
\toprule

\multicolumn{7}{l}{\textbf{CBSD68 (68 images), denoiser GS-DRUNet (RED algorithms) / Prox-DRUNet (PnP algorithms)}}\\
 & \multicolumn{2}{c}{Inpainting} & \multicolumn{2}{c}{Tomography (60 views)} & \multicolumn{2}{c}{SR $\times 4$}\\
 & constant $\sigma$ & annealed & constant $\sigma$ & annealed & constant $\sigma$ & annealed\\ \midrule
RED--GD & 29.71 & 30.99 & 27.93 & \textbf{29.02} & 24.64 & 24.70\\
PnP--PGD & 30.30 & 30.70 & 23.58 & 27.42 & 24.68 & 24.70\\
SNORE & 29.93 & 30.93 & 27.89 & 28.98 & 24.64 & \textbf{24.73}\\
SNOPnP & 30.26 & 30.64 & 22.40 & 27.33 & 24.65 & 24.69\\
ERED & 29.40 & \textbf{31.07} & 27.88 & 28.96 & 24.62 & 24.70\\
DPIR (DRUNet) & \multicolumn{2}{c}{29.15} & \multicolumn{2}{c}{28.35} & \multicolumn{2}{c}{24.35}\\
DPIR (GS-DRUNet) & \multicolumn{2}{c}{29.19} & \multicolumn{2}{c}{28.79} & \multicolumn{2}{c}{24.38}\\
\bottomrule
\end{tabular}

\end{table}

\begin{figure}[h]
\centering
\includegraphics[width=\linewidth]{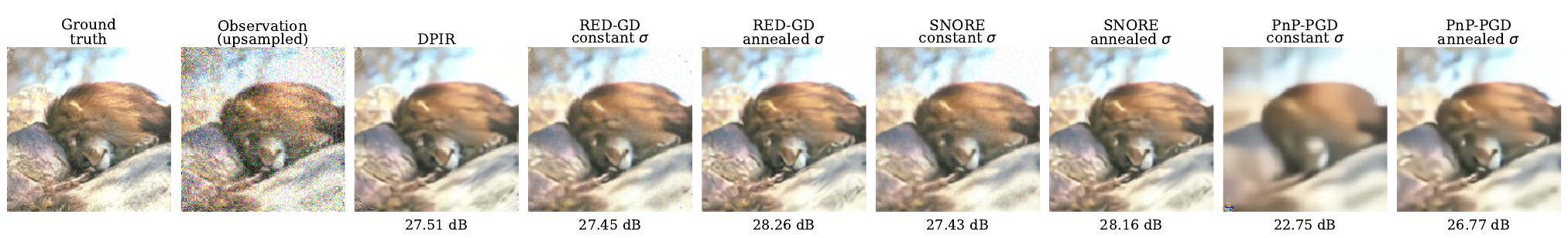}\\[2pt]
\includegraphics[width=\linewidth]{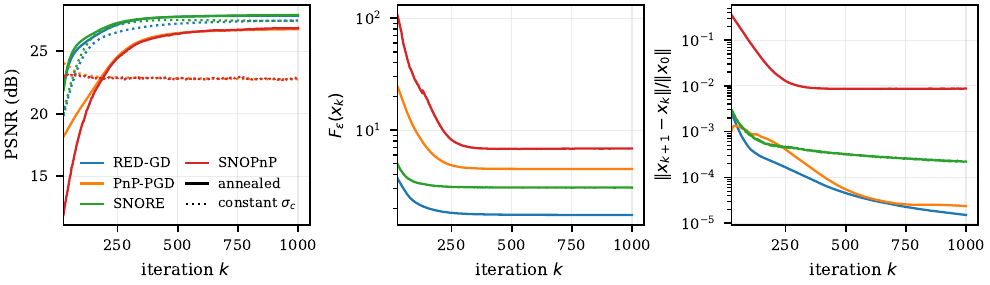}
\caption{Tomography with $60$ views on CBSD68, same layout as Figure~\ref{fig:hole}; the observation is shown through its filtered back-projection.}
\label{fig:tomo}
\end{figure}

\begin{figure}[h]
\centering
\includegraphics[width=\linewidth]{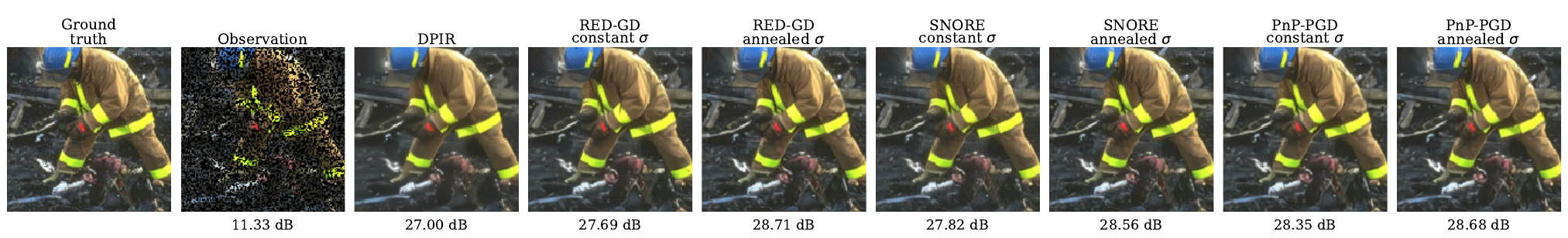}\\[2pt]
\includegraphics[width=\linewidth]{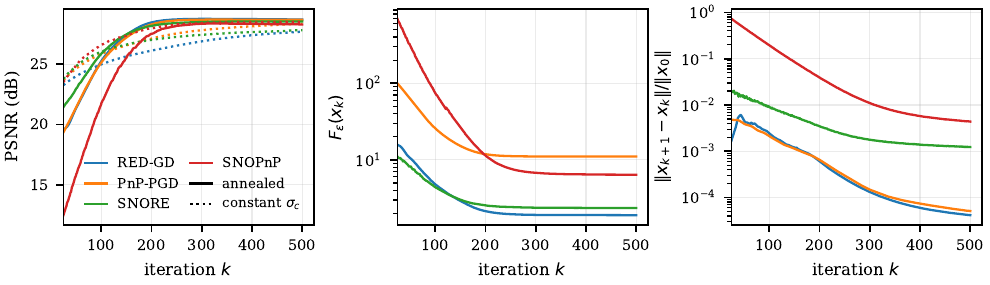}
\caption{Inpainting of $50\%$ random missing pixels on CBSD68, same layout as Figure~\ref{fig:hole}.}
\label{fig:inpaint}
\vspace{-0.5cm}
\end{figure}

\begin{table}[h]
\centering\scriptsize
\setlength{\tabcolsep}{3pt}
\caption{Selected hyper-parameters of the different algorithms for each problem and dataset (validation on $16$ images)}
\label{tab:hparams}
{\footnotesize
\begin{tabular}{llccccc}
\toprule
dataset & problem & algorithm & $\varepsilon$ ($\sigma_c$) & $\lambda$ & step & $q$\\
\midrule
FFHQ & Inpainting & RED--GD & 0.014 (2) & 0.3 & $\gamma_{\rm rel}=1$ & --\\
FFHQ & Inpainting & PnP--PGD & 0.01 (0.04) & -- & $c=0.3$ & --\\
FFHQ & Inpainting & SNORE & 0.01 (2) & 0.3 & $\gamma_{\rm rel}=1$ & 0.55\\
FFHQ & Inpainting & SNOPnP & 0.01 (0.16) & -- & $c=0.95$ & 1\\
FFHQ & Inpainting & ERED & 0.01 (2) & 0.5 & $\gamma_{\rm rel}=1$ & 1\\
\midrule
FFHQ & Demosaicing & RED--GD & 0.01 (0.08) & 0.1 & $\gamma_{\rm rel}=1$ & --\\
FFHQ & Demosaicing & PnP--PGD & 0.01 (0.04) & -- & $c=0.5$ & --\\
FFHQ & Demosaicing & SNORE & 0.01 (0.04) & 0.3 & $\gamma_{\rm rel}=1$ & 0.55\\
FFHQ & Demosaicing & SNOPnP & 0.01 (0.028) & -- & $c=0.5$ & 1\\
FFHQ & Demosaicing & ERED & 0.01 (0.08) & 1 & $\gamma_{\rm rel}=1$ & 1\\
\midrule
FFHQ & SR $\times 4$ & RED--GD & 0.08 (0.2) & 0.01 & $\gamma_{\rm rel}=1$ & --\\
FFHQ & SR $\times 4$ & PnP--PGD & 0.04 (0.057) & -- & $c=0.3$ & --\\
FFHQ & SR $\times 4$ & SNORE & 0.056 (0.11) & 0.01 & $\gamma_{\rm rel}=1$ & 1\\
FFHQ & SR $\times 4$ & SNOPnP & 0.08 (0.08) & -- & $c=0.95$ & 0.7\\
FFHQ & SR $\times 4$ & ERED & 0.08 (0.2) & 0.01 & $\gamma_{\rm rel}=1$ & 1\\
\midrule
CBSD68 & Inpainting & RED--GD & 0.01 (0.03) & 0.3 & $\gamma_{\rm rel}=1$ & --\\
CBSD68 & Inpainting & PnP--PGD & 0.02 (0.02) & -- & $c=0.95$ & --\\
CBSD68 & Inpainting & SNORE & 0.01 (0.06) & 0.3 & $\gamma_{\rm rel}=1$ & 0.55\\
CBSD68 & Inpainting & SNOPnP & 0.01 (0.03) & -- & $c=0.7$ & 1\\
CBSD68 & Inpainting & ERED & 0.01 (0.03) & 0.3 & $\gamma_{\rm rel}=1$ & 1\\
\midrule
CBSD68 & Tomography & RED--GD & 0.028 (0.057) & 0.03 & $\gamma_{\rm rel}=1$ & --\\
CBSD68 & Tomography & PnP--PGD & 0.01 (0.04) & -- & $c=0.95$ & --\\
CBSD68 & Tomography & SNORE & 0.04 (0.08) & 0.03 & $\gamma_{\rm rel}=1$ & 0.55\\
CBSD68 & Tomography & SNOPnP & 0.01 (0.057) & -- & $c=0.7$ & 0.7\\
CBSD68 & Tomography & ERED & 0.028 (0.057) & 0.03 & $\gamma_{\rm rel}=1$ & 0.55\\
\midrule
CBSD68 & SR $\times 4$ & RED--GD & 0.02 (0.04) & 0.1 & $\gamma_{\rm rel}=1$ & --\\
CBSD68 & SR $\times 4$ & PnP--PGD & 0.028 (0.04) & -- & $c=0.7$ & --\\
CBSD68 & SR $\times 4$ & SNORE & 0.04 (0.056) & 0.03 & $\gamma_{\rm rel}=1$ & 1\\
CBSD68 & SR $\times 4$ & SNOPnP & 0.02 (0.04) & -- & $c=0.5$ & 1\\
CBSD68 & SR $\times 4$ & ERED & 0.02 (0.04) & 0.1 & $\gamma_{\rm rel}=1$ & 0.55\\
\bottomrule
\end{tabular}}

\vspace{-0.5cm}
\end{table}

\end{document}